\documentclass[a4paper]{amsart}

\usepackage[T1]{fontenc}
\usepackage{amsmath, amssymb, amsthm, mathrsfs, mathtools, marvosym, graphicx}
\usepackage{varioref}
\usepackage[backref=page]{hyperref}
\usepackage{enumitem, xspace, ifthen, comment}
\usepackage[all]{xy}
\usepackage[dvipsnames]{xcolor}
\usepackage[nameinlink, capitalize]{cleveref}
\usepackage{tikz}
\usetikzlibrary{shapes, calc, decorations.pathreplacing}
\usepackage{tikz-cd}
\tikzset{commutative diagrams/arrow style=Latin Modern}

\setlist[enumerate]{
  label=(\thethm.\arabic*),
  before={\setcounter{enumi}{\value{equation}}},
  after={\setcounter{equation}{\value{enumi}}},
  labelindent=0cm,
  leftmargin=*,
  widest=iiii,
  itemsep=1ex
}

\setlist[itemize]{
  leftmargin=*,
  itemsep=1ex,
  label=$\circ$
}

\allowdisplaybreaks

\newtheorem*{thm-plain}{Theorem}
\newtheorem{thm}{Theorem}[section]
\newtheorem{lem}[thm]{Lemma}
\newtheorem{prp}[thm]{Proposition}
\newtheorem{cor}[thm]{Corollary}

\newtheorem*{boi}{One Sacrilegious Conjecture}
\newtheorem{ques}[thm]{Question}
\numberwithin{equation}{thm}

\newtheorem{thmalph}{Theorem}

\theoremstyle{definition}
\newtheorem{dfn}[thm]{Definition}
\newtheorem{prpdfn}[thm]{Proposition/Definition}
\newtheorem*{dfn-plain}{Definition}

\theoremstyle{remark}
\newtheorem{clm}[thm]{Claim}

\newtheorem{awlog}[thm]{Additional Assumption}
\newtheorem{ntn}[thm]{Notation}

\newtheorem{rem}[thm]{Remark}

\newtheorem{exm}[thm]{Example}
\newtheorem*{rem-plain}{Remark}

\newcommand{\inv}{^{-1}}
\newcommand{\from}{\colon}
\newcommand{\gdw}{\ensuremath{\Leftrightarrow}}
\newcommand{\imp}{\ensuremath{\Rightarrow}}
\newcommand{\lto}{\longrightarrow}

\newcommand{\x}{\times}
\newcommand{\inj}{\hookrightarrow}

\newcommand{\isom}{\cong}
\newcommand{\defn}{\coloneqq}

\newcommand{\tensor}{\otimes}
\newcommand{\id}{\mathrm{id}}

\newcommand{\wt}{\widetilde}
\newcommand{\wb}{\overline}

\renewcommand{\d}{\mathrm d}

\newcommand{\ddual}{^{\smash{\scalebox{.7}[1.4]{\rotatebox{90}{\textup\guilsinglleft} \hspace{-.5em} \rotatebox{90}{\textup\guilsinglleft}}}}}

\newcommand{\factor}[2]{\left. \raise 2pt\hbox{$#1$} \right/\hskip -2pt \raise -2pt\hbox{$#2$}}

\newdir{ ir}{{}*!/-5pt/@^{(}} 
\newdir{ il}{{}*!/-5pt/@_{(}} 

\newenvironment{Mat}[1]{\begin{pmatrix}#1\end{pmatrix}}

\newenvironment{sMat}[1]{\left(\begin{smallmatrix}#1\end{smallmatrix}\right)}

\DeclareMathOperator{\coker}{coker}

\DeclareMathOperator{\Sym}{Sym}

\newcommand{\set}[1]{\left\{ #1 \right\}}

\def\rd#1.{\lfloor{#1}\rfloor}
\def\rp#1.{\lceil{#1}\rceil}
\def\tw#1.{\langle{#1}\rangle}

\renewcommand{\O}[1]{\mathscr{O}_{#1}}
\newcommand{\Omegap}[2]{\Omega_{#1}^{#2}}
\newcommand{\Omegar}[2]{\Omega_{#1}^{[#2]}}
\newcommand{\Omegal}[3]{\Omega_{#1}^{#2} \big( \!\log #3 \big)}
\newcommand{\Omegarl}[3]{\Omega_{#1}^{[#2]} \big( \!\log #3 \big)}

\newcommand{\can}[1]{\omega_{#1}}

\newcommand{\Reg}[1]{{#1}_{\mathrm{reg}}}
\newcommand{\Sing}[1]{{#1}_{\mathrm{sg}}}

\newcommand{\snc}[1]{{#1}_{\mathrm{snc}}}
\newcommand{\codim}[2]{\mathrm{codim}_{#1}(#2)}
\newcommand{\Diff}[2]{\mathrm{Diff}_{#1}(#2)}

\def\Hnought#1.#2.{\mathit{\Gamma} \!\left( #1, #2 \right)}
\def\HH#1.#2.#3.{\mathrm{H}^{#1} \!\left( #2, #3 \right)}
\def\euler#1.#2.{\chi \!\left( #1, #2 \right)}
\def\HHbig#1.#2.#3.{\mathrm{H}^{#1} \!\big( #2, #3 \big)}
\def\hh#1.#2.#3.{h^{#1} \!\left( #2, #3 \right)}
\def\RR#1.#2.#3.{R^{#1} #2_* #3}
\def\HHc#1.#2.#3.{\mathrm{H}_{\mathrm{c}}^{#1} \!\left( #2, #3 \right)}
\def\Hh#1.#2.#3.{\mathrm{H}_{#1} \!\left( #2, #3 \right)}
\def\Hom#1.#2.#3.{\mathrm{Hom}_{#1} \!\left( #2, #3 \right)}
\def\sHom#1.#2.{\mathscr{H}\!om \!\left( #1, #2 \right)}
\def\Ext#1.#2.#3.{\mathrm{Ext}^{#1} \!\left( #2, #3 \right)}
\def\sExt#1.#2.#3.{\mathscr{E}\!xt^{#1} \!\left( #2, #3 \right)}

\newcommand{\A}[1]{\mathbb A^{#1}}
\newcommand{\PP}[1]{\mathbb P^{#1}}

\newcommand{\kahler}{K{\"{a}}hler\xspace}

\DeclareMathOperator{\Spec}{Spec}

\DeclareMathOperator{\Gal}{Gal}

\DeclareMathOperator{\Exc}{Exc}
\DeclareMathOperator{\res}{res}
\DeclareMathOperator{\restr}{restr}

\DeclareMathOperator{\supp}{supp}

\newcommand{\germ}[2]{\left( #1 \in #2 \right)} 
\newcommand{\zp}{\ensuremath{\Z_{(p)}}}
\newcommand{\et}{Regular Extension Theorem\xspace}
\newcommand{\lext}{Logarithmic Extension Theorem\xspace}

\newcommand{\eps}{\varepsilon}
\renewcommand{\theta}{\vartheta}
\renewcommand{\phi}{\varphi}

\newcommand{\N}{\ensuremath{\mathbb N}}
\newcommand{\Z}{\ensuremath{\mathbb Z}}
\newcommand{\Q}{\ensuremath{\mathbb Q}}

\renewcommand{\frm}{\mathfrak m}

\newcommand{\sA}{\mathscr A} \newcommand{\sB}{\mathscr B} \newcommand{\sC}{\mathscr C}
 \newcommand{\sE}{\mathscr E} 
  
  \newcommand{\sL}{\mathscr L}
  
 \newcommand{\sQ}{\mathscr Q}

  \newcommand{\cC}{\ensuremath{\mathcal C}}

\definecolor{forrest}{RGB}{81,133,49}
\definecolor{mydarkblue}{RGB}{10,92,153}

\newcommand{\PreprintAndPublication}[2]{ %
  \sideremark{ %
    \begin{color}{mydarkblue} Preprint \end{color}/ %
    \begin{color}{forrest} Publication \end{color}
  } %
  \begin{color}{mydarkblue} #1 \end{color} %
  \begin{color}{forrest} #2 \end{color} %
  \sideremark{ %
    End of
    \begin{color}{mydarkblue} Preprint \end{color}/ %
    \begin{color}{forrest} Publication \end{color} %
  }
}

\renewcommand{\PreprintAndPublication}[2]{#1}

\title[Differential forms in positive characteristic]{Differential forms on log canonical spaces \\ in positive characteristic}

\author{Patrick Graf}
\address{Lehrstuhl f\"ur Mathematik I, Universit\"at Bayreuth, 95440 Bayreuth, Germany}
\email{\href{mailto:patrick.graf@uni-bayreuth.de}{patrick.graf@uni-bayreuth.de}}
\urladdr{\href{http://www.pgraf.uni-bayreuth.de/en/}{www.graficland.uni-bayreuth.de}}

\date{January 17, 2022}
\thanks{The author was supported in full by a DFG Research Fellowship.}

\hypersetup{
  pdfauthor={Patrick Graf},
  pdftitle={Differential forms on log canonical spaces in positive characteristic},
  pdfkeywords={Reflexive differentials, extension theorem, surfaces in positive characteristic, residue sequence, restriction sequence},
  pdfstartview={Fit},
  pdfpagelayout={OneColumn},
  pdfpagemode={UseNone},
  linktoc=all,
  breaklinks,
  linkcolor=[RGB]{0 0 96},
  citecolor=[RGB]{96 0 0},
  urlcolor=[RGB]{0 96 0},
  colorlinks}

\begin{document}

\begin{abstract}
Deswegen zur{\"{u}}ck zum echten Leben.
In dem, wenn uns Ruhe umgibt, eine Erinnerung hochkommen kann.
Und die k{\"{u}}ndigt sich leise an, wiederholt sich ein paar Mal und man fragt sich: ist das jetzt wirklich so gewesen oder doch anders?
Aber es war schmerzhaft, diese Erinnerung.
Und die kommt immer n{\"{a}}her, und wird immer erlebbarer, und dann ist es pl{\"{o}}tzlich so, als w{\"{a}}re es ganz aktuell, als w{\"{u}}rde es wieder durch einen durchgehen im Hier und Jetzt.
\end{abstract}

\maketitle

\begingroup
\hypersetup{linkcolor=black}
\tableofcontents
\endgroup


\thispagestyle{empty}

\section{Introduction}

Differential forms play an essential r{\^{o}}le in the study of algebraic varieties.
Given an algebraic variety $X$ over a field $k$ and a resolution of singularities $\pi \from Y \to X$, it is natural to ask whether any $p$-form on the regular locus $\Reg X$ extends to a regular $p$-form on $Y$.
There is also a version of this question which concerns pairs and allows certain logarithmic poles.
In order to fix our terminology once and for all, we introduce the following language.
(For notation, see \cref{sec notation}).

\begin{dfn-plain}[Extension properties for differential forms]
Let $(X, D)$ be a pair (i.e.~$X$ is normal and $D$ is a Weil \Q-divisor with coefficients in $[0, 1] \cap \Q$) defined over a field $k$, and $1 \le q \le \dim X$ an integer.
\begin{itemize}
\item We say that \emph{$(X, D)$ satisfies the \et for $q$-forms} if for any proper birational map $\pi \from Y \to X$ from a normal variety $Y$, the natural inclusion
\[ \pi_* \Omegar{Y/k}q \xhookrightarrow{\quad} \Omegar{X/k}q \]
is an isomorphism.
Equivalently, the sheaf $\pi_* \Omegar{Y/k}q$ is reflexive.
It is sufficient to check this for a resolution of singularities $Y \to X$ (if available): cf.~\cite[Lemma~2.13]{GKK10} and note that the proof given there is independent of the base field.
\item We say that \emph{$(X, D)$ satisfies the \lext for $q$-forms} if for any map $\pi$ as above, with $D_Y$ the strict transform of $D$ and $E \subset Y$ the reduced divisorial part of the exceptional set $\Exc(\pi)$, the natural inclusion
\[ \pi_* \Omegarl{Y/k}q{\, \rd D_Y. + E} \xhookrightarrow{\quad} \Omegarl{X/k}q{\, \rd D.} \]
is an isomorphism.
Equivalently, the sheaf $\pi_* \Omegarl{Y/k}q{\, \rd D_Y. + E}$ is reflexive.
Again, it is sufficient to check this for a log resolution $Y \to X$ of $(X, D)$.
\item We say that \emph{$(X, D)$ satisfies the \et} if it satisfies the \et for $q$-forms, for all values of $q$.
Ditto for the logarithmic variant.
\end{itemize}
\end{dfn-plain}

Over the complex numbers, the problem of when the Extension Theorems hold has a long history.
It has been studied by several people using different methods---the following list is not exhaustive: \cite{SvS85, Flenner88, Nam01, deJongStarr04, GKK10, GKKP11}.
The paper mentioned last, \cite{GKKP11}, can in many ways be seen as the culmination\footnote{Very recently, it has been generalized further in~\cite{KebekusSchnell18}, using perverse sheaves.} of this line of research.
It proved the following:
\stepcounter{thm}
\begin{enumerate}
\item\label{gkkp klt} Any complex klt (= Kawamata log terminal) pair $(X, D)$ satisfies the \et~\cite[Thm.~1.4]{GKKP11}.
\item\label{gkkp lc} Any complex log canonical pair $(X, D)$ satisfies the \lext~\cite[Thm.~1.5]{GKKP11}.
\end{enumerate}
Given the importance of these results, it is not free of interest to ask whether similar results also hold in positive characteristic.
Curiously enough, no research in this direction has been conducted so far.
We have identified two main reasons for this:
\begin{itemize}
\item It has been known to experts for some time that~\labelcref{gkkp klt} fails in a strong sense in positive characteristic.
In fact, over any field of nonzero characteristic, there exists a strongly $F$-regular (in particular, klt) surface $X$ violating the \et~(\cref{lext no et}).
\item The proof of~\labelcref{gkkp lc} relies on rather subtle Hodge-theoretic vanishing theorems for Du Bois spaces.
These are either false or not known in positive characteristic, inextricably linking the proof to the complex numbers.
The same can be said of the techniques in~\cite{KebekusSchnell18}.
\end{itemize}
The purpose of this article is to overcome these obstacles, at least for surfaces (but see \cref{lext p fail} for higher dimensions).
Concerning the first issue, our approach is pretty straightforward: as~\labelcref{gkkp klt} fails, we instead concentrate on~\labelcref{gkkp lc}.
(Cf.~however \cref{main reg}, which explores the failure of~\labelcref{gkkp klt} more thoroughly.)
To deal with the second problem, we develop a completely novel and much more hands-on approach to extension.
Our first main result is as follows.

\begin{thm}[Logarithmic Extension for surfaces] \label{main lc}
Let $(X, D)$ be a log canonical surface pair over a perfect field $k$ of characteristic $p \ge 7$.
Then $(X, D)$ satisfies the \lext.
\end{thm}

Our second main result explains when the \lext does imply the \et.

\begin{thm}[Regular Extension for surfaces] \label{main reg}
Let $\germ0{X, D}$ be a surface singularity over a field $k$ of characteristic $p > 0$.
Assume that for some (not necessarily log) resolution $\pi \from Y \to X$, with exceptional curves $E_1, \dots, E_\ell$, the determinant of the intersection matrix $(E_i \cdot E_j)$ is not divisible by $p$.
Then if $\germ0{X, D}$ satisfies the \lext for $1$-forms, it also satisfies the \et for $1$-forms.
\end{thm}

We would like to emphasize the advantages of our approach over the existing techniques.
First of all, we feel that our proof offers a new level of both transparency and tangibility, as it does not explicitly use any Hodge theory (it does, however, rely on the Minimal Model Program).
Secondly, this very same feature also makes it, to a large extent, insensitive to the characteristics of the ground field.
In fact, aside from some effortless changes our approach also yields a new proof of the characteristic zero extension theorem~\cite[Thm.~1.5]{GKKP11}---the details are worked out in \cref{ext higher-dim zero}.
Thirdly and maybe most importantly, we obtain a lucid explanation of why the \lext fails in low characteristics, even for surface rational double points.

\subsection*{Further results in this paper}

Apart from the above extension results, we establish residue and restriction sequences for reflexive differential forms on dlt pairs in positive characteristic, and symmetric powers thereof.
This is analogous to known results in characteristic zero~\cite{GKKP11, Gra12}.
However, it is important to note that actually a slightly stronger notion is required, called \emph{tamely dlt} in this paper.
A dlt pair $(X, D)$ is tamely dlt if $D$ is reduced and the Cartier index of $K_X + D$ is not divisible by $p$ (\cref{tamely dlt}).

The precise statement is as follows.
Even though we only use it as a technical tool in the proof of our main result, we believe that it is of independent interest.

\begin{thm}[Residue sequence] \label{res}
Let $(X, D)$ be a tamely dlt surface pair (in particular, $D$ is reduced), and let $P \subset D$ be an irreducible component.
Set $P^c \defn \Diff P{D - P}$, so that $(K_X + D)\big|_P = K_P + P^c$.
Then there is a short exact sequence
\begin{equation} \label{res seq}
0 \lto \Omegarl X1{D-P} \lto \Omegarl X1D \xrightarrow{\;\res_P\;} \O P \lto 0
\end{equation}
which on the snc locus of $(X, D)$ agrees with the usual residue sequence.
Its restriction to $P$ induces a short exact sequence\footnote{Here, of course, in the middle term we are taking the double dual on $P$ and not on $X$ (the latter would be zero).}
\begin{equation} \label{res P seq}
0 \lto \Omegal P1{\rd P^c.} \lto \Omegarl X1D \big|_P \ddual \xrightarrow{\;\res_P^P\;} \O P \lto 0.
\end{equation}
More generally, for every $m \in \N$ there is a surjective map
\[ \res_P^m \from \Sym^{[m]} \Omegarl X1D \lto\mathrel{\mkern-25mu}\lto \O P \]
which generically coincides with the $m$-th symmetric power of the residue map.
\end{thm}

\begin{thm}[Restriction sequence] \label{restr}
Notation as above.
Then there is a short exact sequence\footnote{By definition, $\Omegarl X1D(-P) \ddual$ means the double dual of $\Omegarl X1D \tensor \O X(-P)$. Taking the reflexive hull is necessary because $P \subset X$ is in general not a Cartier divisor.}
\begin{equation} \label{restr seq}
0 \lto \Omegarl X1D(-P) \ddual \lto \Omegarl X1{D-P} \xrightarrow{\;\restr_P\;} \Omegal P1{\rd P^c.} \lto 0
\end{equation}
which on the snc locus of $(X, D)$ agrees with the usual restriction sequence.
More generally, for every $m \in \N$ there is a surjective map
\[ \restr_P^m \from \Sym^{[m]} \Omegarl X1{D-P} \lto\mathrel{\mkern-25mu}\lto \O P \big( mK_P + \rd mP^c. \big) \]
which generically coincides with the $m$-th symmetric power of the restriction map.
\end{thm}

\subsection*{Sharpness of results}

In \cref{sharp}, we have gathered a number of examples to show that our results are sharp.
First of all, \cref{main lc} does fail in characteristic less than seven, even if $k$ is algebraically closed, $D = 0$ and $X$ is a rational double point (RDP).
More precisely, we show by explicit calculation that the singularity given by the equation $z^2 + x^3 + y^5 = 0$ violates the \lext over \emph{any} field of characteristic $p \le 5$.
In the terminology of Artin's classification of RDPs~\cite{Artin77}, this is the $E_8^0$ singularity.
This failure also occurs for some singularities of types~$D_n$ ($p = 2$) and $E_6, E_7$ ($p = 2, 3$).
We have omitted those calculations, as they are very similar in spirit to the $E_8$ case.

Turning to \cref{main reg}, its statement is sharp too, as shown by the example of contracting a smooth rational curve with self-intersection $-p$ in any characteristic $p > 0$.
In this case, the \lext holds for $1$-forms, but the \et does not.
Again, this can be seen via explicit computation.

The latter example can also be elaborated upon to show that \cref{restr} fails for dlt pairs that are not tamely dlt.
If one tries to run the proof of \cref{main lc} on, say, a $D_n$ singularity in characteristic two, the lack of a suitable restriction map is exactly where the argument breaks down: already the first contraction performed by the MMP produces a pair that is not tamely dlt.
This should be seen as the deeper reason for the failure of \cref{main lc} in low characteristics.

\subsection*{Higher dimensions}

For the majority of readers, a most pressing question will be to what extent \cref{main lc} carries over to higher dimensions.
As we will see in \cref{ext higher-dim zero}, in characteristic zero the higher-dimensional \lext is intimately linked to the fact that on a projective snc pair $(X, D)$, a line bundle $\sL \subset \Omegal XpD$ cannot be big unless $p = \dim X$.
This is the content of the Bogomolov--Sommese vanishing theorem~\cite[Cor.~6.9]{EV92}, while the weaker statement that $\sL$ cannot be ample is a special case of (Kodaira--Akizuki--)Nakano vanishing~\cite[Thm.~1$''$]{AkizukiNakano54}.
Both results fail badly in positive characteristic and in fact there are counterexamples strong enough to show that \cref{main lc} itself does not hold.
The precise statement is as follows and the details of the construction can be found in \cref{ext higher-dim p}.

\begin{thm}[Failure of the higher-dimensional \lext] \label{lext p fail}
Fix an algebraically closed field $k$ of characteristic~$p > 0$.
\begin{enumerate}
\item\label{lext p fail lc} In any dimension $n \ge p - 1$, there exists a log canonical pair $(X, \emptyset)$ over~$k$ that violates the \lext for $(n - 2)$-forms.
\item\label{lext p fail can} If $n \ge 2p - 1$, there exists a \emph{canonical} pair $(X, \emptyset)$ for which the \lext fails as above.
\item\label{lext p fail term} If $n \ge 3p - 1$, there even exists a \emph{terminal} pair $(X, \emptyset)$ as above.
\end{enumerate}
Furthermore, the above examples admit log resolutions.
\end{thm}

As \cref{main lc} already fails for surfaces if the characteristic is low, \cref{lext p fail} becomes interesting only for $p \ge 7$.
In this sense, the lowest-dimensional example it provides is a $6$-dimensional singularity in characteristic~$7$.
The following conjecture hence remains open.

\begin{boi}
Over a perfect field of characteristic $p \ge 7$, the \lext holds for log canonical pairs of dimension $\le p - 2$.
\end{boi}

Of course, we do not believe in the Sacrilegious Conjecture.
Rather, our inability to disprove it is caused by a lack of techniques to produce meaningful counterexamples.

\subsection*{Relation to $F$-singularities}

The examples in \cref{lext p fail} are (un-)fortunately not $F$-pure.
On the other hand, using classification results~\cite[Thm.~1.1]{Hara98} one can show that all normal $F$-regular surface singularities over a perfect field satisfy the \lext.
The same is probably true for $F$-pure surfaces, but the case distinctions get much more tedious.
These observations have led us to the following intriguing question:

\begin{ques} \label{565}
Is there a version of the \lext for strongly $F$-regular/$F$-pure singularities that does not exclude low characteristics and works in any dimension?
\end{ques}

The following line of attack appears to be quite promising.
By~\cite[Thm.~3.3]{Watanabe91}, the affine cone over a smooth projective variety $X$ is $F$-pure if and only if $X$ is (globally) $F$-split.
Hence one would need to investigate whether $F$-split varieties satisfy Nakano vanishing.
Since at least Kodaira vanishing obviously holds for these, chances may not be that bad.
This would immediately provide a positive answer to \cref{565} for cones.
On the other hand, if Nakano vanishing failed, we would obtain an $F$-pure counterexample to the \lext.

\PreprintAndPublication{
\subsection*{Outline of proof}

We would like to explain the general strategy for the proof of \cref{main lc}.
It can be broken up into three major steps.

\subsubsection*{Step 1: Baby case}

The basic idea is quite simple.
Consider an lc surface pair $(X, D)$ as in the theorem, where for simplicity we will assume $D = 0$.
Also, fix a log resolution $\pi \from Y \to X$, with exceptional divisor $E$.
Given any reflexive $1$-form $\sigma \in \HH0.X.\Omegar X1.$, we may extend $\sigma$ to a rational logarithmic $1$-form $\wt\sigma$ on $Y$, i.e.~a global section of $\Omegal Y1E(G)$, where $G \ne 0$ is effective and $\supp(G) \subset E$.
This is equivalent to giving a map $\O Y(-G) \to \Omegal Y1E$.
We use the fact that $G^2 < 0$ combined with the residue sequence on the snc pair $(Y, E)$ to show that this map factors as $\O Y(-G) \to \Omegal Y1{E-P}$, for some component $P \subset E$.
The same argument repeated, but this time using the restriction sequence, then shows that our map actually even factors as $\O Y(-G) \to \Omegal Y1E(-P)$.
This means that $G$ may be replaced by the strictly smaller divisor $G - P$.
Repeating this procedure finitely many times, namely as long as $G$ is nonzero, we finally obtain $G = 0$ and hence $\wt\sigma \in \HH0.Y.\Omegal Y1E.$ as desired.

It turns out that for this approach to work smoothly, $-(K_Y + E)$ needs to be $\pi$-nef.
By adjunction, $(K_Y + E)\big|_P = K_P + (E - P)\big|_P$ for every component $P \subset E$ and hence the nefness condition in practice means that $E$ is either a chain or a cycle of rational curves, or a single elliptic curve.
Curiously, this already implies \cref{main lc} (in any characteristic!) for two extreme cases: the $A_n$ singularities on the one hand and Gorenstein log canonical singularities \emph{that are not canonical} on the other hand.

\subsubsection*{Step 2: General case}

For e.g.~a $D_4$ singularity, the baby case argument breaks down and this is where the technical complications, as well as the restrictions on the ground field, start.
Indeed, the crucial idea is to use the Minimal Model Program to factor the resolution $\pi$ into a series of steps each of which satisfies the nefness condition from Step~1, as detailed in \cref{sec fact res}.
As is well-known, if run on an snc pair, the MMP will produce intermediate steps and an end result that are only dlt.
This is where Theorems~\labelcref{res} and~\labelcref{restr} come into play.
They are proved in \cref{res dlt}, with preparations in \cref{adj diff dlt}.

With these technical generalizations in place, the ideas from Step~1 apply to show that $1$-forms extend along each step in the factorization of $\pi$ provided by the MMP.
The precise statement may be found in \cref{lift elem}.
Piecing all the steps together, we can prove our result \emph{if $(X, D)$ has a tame resolution}, i.e.~one that factors in such a way that all intermediate pairs are tamely dlt.
It is easy to see that this implies the following weak form of \cref{main lc}:
\emph{For every extended dual graph $\Gamma$ (i.e.~incorporating self-intersection numbers as well as the boundary components), there is a prime number $p_0 = p_0(\Gamma)$ such that every log canonical surface pair with dual graph $\Gamma$ and defined over a field of characteristic $p \ge p_0$ satisfies the \lext.}

\subsubsection*{Step 3: Effective bounds}

Two things remain to be done: First, to eliminate the dependency of $p_0$ on $\Gamma$ and second, to give an effective value for $p_0$.
In order to achieve this, we resort to the classification of log canonical surface pairs over an algebraically closed field.
(This is also where the perfectness hypothesis on $k$ comes from: the base change to the algebraic closure needs to be separable.)
We stress that this is the only place in the whole paper where classification is used.
It turns out that we may choose $p_0 = 7$, finishing the proof of \cref{main lc}.
The details are contained in \cref{sec main lc}.
}{}

\subsection*{Acknowledgements}

I would like to thank Karl Schwede and Thomas Polstra for helpful discussions and answering many of my questions.
Jorge Vit\'orio Pereira has brought the paper~\cite{Kol95nh} to my attention via MathOverflow.
The Department of Mathematics at the University of Utah has provided support and excellent working conditions.

I am particularly grateful to the anonymous referee for extremely diligent proofreading, which has not only greatly improved the presentation of certain details, but also led to some simplifications.

\section{Notation and conventions} \label{sec notation}

\subsection*{Base field}

Throughout this paper, we work over a field $k$, which except for \cref{ext higher-dim zero} will be assumed to be of positive characteristic~$p > 0$.
Further assumptions (perfect, algebraically closed, \dots) will be expressly stated whenever necessary.

\subsection*{Pairs and divisors}

A \emph{pair} $(X, D)$ consists of a normal variety $X$ and a Weil \Q-divisor $D = \sum a_i D_i$ with coefficients $0 \le a_i \le 1$.
The pair is called \emph{reduced} if $D$ is reduced.
The \emph{round-down} of $D$ is denoted by $\rd D. \defn \sum \rd a_i. D_i$, and similarly for the \emph{round-up} $\rp D.$.
The \emph{fractional part} $\{ D \}$ is, by definition, $D - \rd D.$.
For a uniform definition of the singularities of the MMP (klt, plt, dlt, lc, \dots), we refer to \cite[Def.~2.8]{Kollar13}.

The regular and singular loci of a variety $X$ are denoted $\Reg X$ and $\Sing X$, respectively.
We say that a closed subset $Z \subset X$ is \emph{small} if $\codim XZ \ge 2$, and that an open subset $U \subset X$ is \emph{big} if $X \setminus U$ is small.

A Weil divisor $D$ on a normal variety $X$ is said to be \emph{\zp-Cartier} if it has a multiple not divisible by $p$ which is Cartier.
Equivalently, $D$ is in the image of the natural map
\[ \operatorname{Div}(X) \tensor_\Z \zp \lto \operatorname{WDiv}(X) \tensor_\Z \zp. \]
Since $\Z_{(0)} = \Q$, in characteristic zero we recover the usual notion of being \Q-Cartier.
More generally, the \emph{Cartier index} of $D$ is the smallest integer $m > 0$ with $mD$ Cartier (or $+\infty$ if no such $m$ exists).

\subsection*{Reflexive sheaves}

Let $X$ be a normal variety and $\sE$ a coherent sheaf on $X$.
The $\O X$-double dual (or reflexive hull) of $\sE$ is denoted by $\sE \ddual$.
The sheaf $\sE$ is called \emph{reflexive} if the canonical map $\sE \to \sE \ddual$ is an isomorphism.
A \emph{Weil divisorial sheaf} is a reflexive sheaf of rank one.
A coherent subsheaf $\sA \subset \sE$ of a reflexive sheaf is said to be \emph{saturated} if the quotient $\factor \sE \sA$ is torsion-free.
We use square brackets $^{[-]}$ as an abbreviation for taking the double dual, e.g.~$\sE^{[k]} = (\sE^{\tensor k}) \ddual$ and $f^{[*]} \sE = (f^* \sE) \ddual$ for a map $f \from Y \to X$ with $Y$ normal.

Let $D \subset X$ be a reduced divisor.
Then we denote by
\[ \sE(*D) \defn \varinjlim \big( \sE \tensor \O X(mD) \big) \ddual \]
the quasi-coherent sheaf of sections of $\sE$ with arbitrarily high order poles along~$D$.
If $i \from U \inj X$ is the inclusion of the snc locus of $(X, D)$, the sheaf of \emph{reflexive differential $q$-forms} is defined to be $\Omegarl {X/k}qD \defn i_* \Omegal {U/k}q{D\big|_U}$.
The base field $k$ will usually be dropped from notation.

Following are some useful properties of reflexive sheaves which will be used implicitly or explicitly.
For proofs, we refer to~\cite[Sec.~3]{Gra12}.

\begin{lem}
Let $\sE$ be a reflexive sheaf on the normal variety $X$ and $\sA, \sB \subset \sE$ coherent subsheaves, with $\sA$ saturated.
\begin{enumerate}
\item The sheaf $\sA$ is reflexive.
\item Let $s$ be a rational section of $\sA$ which is regular as a section of $\sE$.
Then $s$ is also regular as a section of $\sA$.
\item\label{gen equal} Suppose that for some dense open subset $U \subset X$, the subsheaves $\sA\big|_U$ and $\sB\big|_U$ of $\sE\big|_U$ are equal.
Then it follows that $\sB \subset \sA$. \qed
\end{enumerate}
\end{lem}

\section{Factorizing resolutions} \label{sec fact res}

It is well-known that in characteristic zero, the MMP can be used to obtain log crepant partial resolutions for log canonical pairs (called ``minimal dlt models'', ``dlt blowups'', or ``dlt modifications'').
See for example~\cite[Thm.~3.1]{KK10}.
Here we would like to point out that the same argument also works for surfaces over arbitrary fields.
The reason is that the MMP for log canonical surfaces is very well developed~\cite{Tanaka18}.
In fact, our proof is even simpler than the one in~\cite{KK10} because we do not have to perturb the dlt pair of interest into a linearly equivalent klt pair.

Unlike~\cite{KK10}, we are not only interested in the end product of the MMP (in the notation below, the map $f$), but also in the intermediate steps.
Note that since we are on a surface, we can use Mumford's pullback to get the same result also for \emph{numerically} log canonical pairs~\cite[Notation~4.1]{KM98}.
This will be important later.

\begin{thm} \label{fact res}
Let $(X, D)$ be a numerically log canonical surface pair and $\pi \from Y \to X$ a log resolution, with exceptional divisor $E$.
Then $\pi$ can be factored into a sequence of maps as follows:
\[ \xymatrix{
Y = Y_0 \ar^-{\phi_0}[r] \ar_\pi[ddrr] & Y_1 \ar^-{\phi_1}[r] & \cdots\cdots\cdots \ar^-{\phi_{r-2}}[r] & Y_{r-1} \ar^-{\phi_{r-1}}[r] & Y_r = Z \ar^-f[ddll] \\
\\
& & X
} \]
such that, setting $\wt D_0 \defn \pi\inv_* D + E$ and $\wt D_{i+1} \defn (\phi_i)_* \wt D_i$, the following properties hold:
\begin{enumerate}
\item\label{fact res.1} For any $0 \le i \le r$, the pair $(Y_i, \wt D_i)$ is dlt and $Y_i$ is \Q-factorial.
\item\label{fact res.2} For any $0 \le i \le r - 1$, the exceptional locus of $\phi_i$ is irreducible.
\item\label{fact res.3} The map $f$ is \emph{(numerically) log crepant}, that is, $K_Z + \wt D_r = f^* (K_X + D)$.
\end{enumerate}
\end{thm}

\begin{proof}
Let $F_1, \dots, F_n$ be all the irreducible components of $E$, and consider the ramification formula $K_Y + \pi\inv_* D = \pi^* (K_X + D) + \sum_{i=1}^n a_i F_i$, where $\pi^*(-)$ denotes Mumford's pullback.
We then have
\begin{align} \label{496}
K_Y + \wt D_0 = \pi^* (K_X + D) + \sum_{i=1}^n (a_i + 1) F_i.
\end{align}
We may run the MMP on the dlt pair $(Y, \wt D_0)$ and obtain a minimal model $\phi \from Y \to Z$ over $X$~\cite[Thm.~1.1]{Tanaka18}.
This provides the maps in the statement to be proven.
Also,~\labelcref{fact res.1} and~\labelcref{fact res.2} are clear by construction.
It remains to show~\labelcref{fact res.3}.
To this end, push forward~\labelcref{496} to $Z$:
\begin{align} \label{504}
K_Z + \wt D_r = f^* (K_X + D) + \underbrace{\phi_* \left( \sum_{i=1}^n (a_i + 1) F_i \right)}_{\text{$f$-nef by construction}}.
\end{align}
The Negativity Lemma~\cite[Lemma~3.40]{KM98} implies that the underbraced term in the above formula is zero.
Hence~\labelcref{504} simplifies to~\labelcref{fact res.3}.
\end{proof}

\section{Adjunction and the different on dlt surface pairs} \label{adj diff dlt}

The different is a correction term that makes the adjunction formula work in the presence of singularities.
For a general treatment of the different, including the case of positive characteristic, see~\cite[Ch.~4]{Kollar13}.
On a surface, things are somewhat simpler, as explained in~\cite[Def.~2.34]{Kollar13}.

\begin{prpdfn}[Different on surfaces]
Let $X$ be a normal \Q-factorial surface and $B \subset X$ a reduced irreducible curve with normalization $\wb B \to B$.
Let $B'$ be a \Q-divisor that has no common components with $B$.
Then there is a canonically defined \Q-divisor $\Diff{\wb B}{B'}$ on $\wb B$, called the \emph{different}, such that
\[ (K_X + B + B')\big|_{\wb B} \; \sim_\Q \; K_{\wb B} + \Diff{\wb B}{B'}. \]
\end{prpdfn}

We will mostly be interested in the case where $(X, B)$ is dlt, in which case $B$ is regular by \cref{dlt structure} below.
Hence $\wb B = B$ and we may write
\[ \Diff B{B'} = \sum_{x \in B} \delta_x \cdot [x], \]
where $\delta_x \ne 0$ only for points $x$ that are singular on $X$ or contained in $\supp B'$.
We need to compute the coefficients $\delta_x$ in relation to the singularities of $(X, B + B')$.
In positive characteristic this is only possible under the following additional tameness hypothesis:

\begin{dfn}[Tamely and fiercely dlt pairs] \label{tamely dlt}
A pair $(X, D)$ over a field of characteristic $p$ is called \emph{tamely dlt} if the following hold:
\begin{enumerate}
\item\label{768} $(X, D)$ is reduced and dlt,
\item\label{595} $K_X + D$ is \zp-Cartier (see \cref{sec notation}).
\end{enumerate}
If Condition~\labelcref{768} is satisfied but~\labelcref{595} is not, the pair is said to be \emph{fiercely dlt}.
\end{dfn}

In the case $p = 0$, we recover the usual notion of a reduced dlt pair.
The main result concerning the different is then as follows.
The reader may like to compare this to~\cite[Thm.~3.36]{Kollar13}, where a similar formula is proven under slightly different assumptions.

\begin{thm}[Computation of the different] \label{compute diff}
Let $(X, D)$ be a tamely dlt surface pair, and let $P \subset D$ be an irreducible component.
Write
\[ \Diff P{D - P} = \sum_{x \in P} \delta_x \cdot [x] \]
as above.
Then, referring to the dichotomy in \cref{dlt structure} below:
\begin{enumerate}
\item\label{diff snc} If locally at $x$,~\labelcref{dlt snc} holds, then $\delta_x = 1$.
\item\label{diff plt} If locally at $x$,~\labelcref{dlt plt} holds, then $\delta_x = 1 - \frac1m$, where $m$ is the Cartier index of $K_X + D$ at $x$.
\end{enumerate}
\end{thm}

\subsection{The local structure of dlt surfaces}

Locally, dlt surface pairs are in some sense quite simple (even if they are fierce):

\begin{prp}[Dichotomy for dlt surfaces] \label{dlt structure}
Let $(X, D)$ be a reduced dlt surface pair, and let $x \in \supp D$ be any point.
Then either one of the following holds:
\begin{enumerate}
\item\label{dlt snc} The pair $(X, D)$ is snc at $x$, and $x$ is contained in exactly two components of $D$.
\item\label{dlt plt} The divisor $D$ is regular at $x$ and the pair $(X, D)$ is plt at $x$.
\end{enumerate}
In particular, every irreducible component of $D$ is regular.
\end{prp}

\begin{proof}
Assume that we are not in case~\labelcref{dlt snc}.
Then either $(X, D)$ is snc at $x$, but $D$ has only one component at $x$.
In this case,~\labelcref{dlt plt} clearly holds.
Or the pair $(X, D)$ is not snc at $x$, in which case it is plt at $x$ by definition.
Regularity of $D$ at $x$ then follows from~\cite[3.35]{Kollar13}.
\end{proof}

In the following corollary, the crucial point is the separability of the maps $\gamma_\alpha$.
Note that the $U_\alpha$ cover only $\supp D$ and not all of $X$.

\begin{cor}[Dlt surfaces as quotients] \label{dlt quotient}
Let $(X, D)$ be a tamely dlt surface pair.
Then there exist finitely many Zariski-open subsets $\set{U_\alpha} _{\alpha \in I}$ of $X$ that cover $\supp D$ and admit maps
\[ \gamma_\alpha \from V_\alpha \to U_\alpha \quad \text{finite quasi-\'etale separable cyclic Galois} \]
such that the pairs $(V_\alpha, \gamma_\alpha^* D)$ are snc for all indices $\alpha \in I$.
\end{cor}

\begin{proof}
Let $x \in \supp D$ be any point, and apply \cref{dlt structure}.
If we are in case~\labelcref{dlt snc}, we may take $\gamma_\alpha = \id$ and there is nothing to show.
In case~\labelcref{dlt plt}, let $\gamma_\alpha$ be a local index one cover with respect to $K_X + D$.
Then $\gamma_\alpha$ by construction has all the properties claimed, except separability.
But separability is also clear because of our assumption that $K_X + D$ is \zp-Cartier.
It remains to see that $(V_\alpha, \gamma_\alpha^* D)$ is snc.
To this end, note that this pair is again plt~\cite[Cor.~2.43]{Kollar13}.
Furthermore, as $K_{V_\alpha} + \gamma_\alpha^* D$ is Cartier, the discrepancies are actually integral and hence non-negative.
The pair $(V_\alpha, \gamma_\alpha^* D)$ is therefore canonical.
Let $y \in V_\alpha$ be the unique point in $\gamma_\alpha\inv(x)$.
Then $y \in \supp \gamma_\alpha^* D$.
The claim now follows from~\cite[Thm.~2.29]{Kollar13}.
\end{proof}

\subsection{Proof of \cref{compute diff}}

Case~\labelcref{diff snc} is clear, hence we concentrate on Case~\labelcref{diff plt}.
We follow the local computational approach as illustrated in~\cite[Ex.~4.3]{Kollar13}.
Let $\gamma \from V \to U$ be a map as in \cref{dlt quotient}, where $x \in U$, and put
\begin{align*}
D_V & = \gamma^* D, \quad \text{a regular curve,} \\
\gamma_D & = \gamma\big|_{D_V}, \\
\sigma & = \text{a local generator for $m \big( K_U + D\big|_U \big)$}, \\
\sigma_V & = \gamma^* \sigma, \\
\omega & = \text{a local generator for $K_V + D_V$}.
\end{align*}
Then $\sigma_V = \omega^{[m]}$ up to a unit, that is, for a suitable choice of $\omega$.
It follows that
\begin{equation} \label{628}
\res(\sigma_V) = \res(\omega^{[m]}) = \res(\omega)^m = \omega'^m,
\end{equation}
where $\omega'$ is a local generator for $\can{D_V}$.
On the other hand, as $\gamma$ is quasi-\'etale and the residue map (in the snc case) commutes with \'etale pullback, we have
\begin{equation} \label{633}
\res(\sigma_V) = \gamma_D^* \res(\sigma).
\end{equation}
Let $t \in \O{D, x}$ be a local parameter of $D$ at $x$, and let $u \in \O{D_V, y}$ be a local parameter of $D_V$ at the unique point $y$ lying over $x$ such that $\omega' = \d u$.
Then $\gamma_D^*(t) = \eps u^m$ for some unit $\eps \in \O{D_V, y}^\times$.
Hence, writing $\res(\sigma) = t^k (\d t)^m$ up to a unit, with $k$ to be determined, combining~\labelcref{628} with~\labelcref{633} gives
\begin{align*}
\eps^k u^{km} (\eps m u^{m-1} \d u + u^m \d \eps)^m & = \big( \eps^{k + m} m^m u^{m(k + m - 1)} + \cdots \big) \cdot (\d u)^m \\
& = (\d u)^m,
\end{align*}
where the dots stand for terms involving higher powers of $u$.
By the tameness assumption, $m \ne 0$ in the ground field and we obtain $m(k + m - 1) = 0$.
So $k = 1 - m$ and $\delta_x = -k/m = 1 - 1/m$, as claimed. \qed

\section{Residues and restriction on dlt surfaces} \label{res dlt}

\stepcounter{thm}
In this section, we prove Theorems~\labelcref{res} and~\labelcref{restr}.

\subsection{Proof of \cref{res}}

The proof is divided into four steps.

\subsubsection*{Step 1: Symmetric residue maps}

First we will construct the maps $\res_P^m$.
So fix a natural number $m$ and consider the $m$-th symmetric power of the residue map on the snc locus of $(X, D)$.
Pushing it forward to all of $X$ yields a map
\begin{equation} \label{743}
\Sym^{[m]} \Omegarl X1D \lto \O P(*\rp P^c.)
\end{equation}
to the sheaf of rational functions on $P$ with arbitrarily high order poles along $\supp \rp P^c.$.
We need to show that~\labelcref{743} factorizes via $\Sym^{[m]} \Omegarl X1D \to \O P$, for this will be the desired map $\res_P^m$.
So let $\sigma$ be an arbitrary local section of $\Sym^{[m]} \Omegarl X1D$, defined on an open set $U \subset X$.
Let $\wt\sigma$ be its image under~\labelcref{743}, and (after possibly shrinking $U$) pick a map $\gamma \from V \to U$ as in \cref{dlt quotient}.

We will employ the following regularity criterion: $\wt\sigma \in \Hnought U.\O P.$ if and only if $\gamma^*(\wt\sigma) \in \Hnought V.\O{P_V}.$, where $P_V \defn \gamma^* P$ and $D_V \defn \gamma^* D$.
This criterion holds because $P$ is regular, in particular normal.
(Recall that if $A \subset B$ is a finite extension of normal domains and $Q(A)$ is the fraction field of $A$, then $B \cap Q(A) = A$.
In our situation, $A$ is a local ring of $P$ and $B$ is a suitable local ring of $P_V$.)

By \cref{dlt quotient}, the pair $(V, D_V)$ is snc, hence $\Omegal V1{D_V}$ is locally free and we obtain a residue map
\[ \res_{P_V}^m \from \Sym^{[m]} \Omegarl V1{D_V} = \Sym^m \Omegal V1{D_V} \lto \Sym^m \O{P_V} = \O{P_V}. \]
Furthermore, note that $\gamma \from (V, D_V) \to (U, D|_U)$ is a ``morphism of logarithmic pairs'' in the sense of~\cite[Def.~2.4]{GKK10} (this simply means that $\gamma\inv(D|_U) = D_V$ set-theoretically).
Therefore, by~\cite[Remark~2.10]{GKK10} we can pullback $\sigma$ to a regular section of $\Sym^m \Omegal V1{D_V}$, at least off the preimage of the non-snc locus $\Sing{(U, D|_U)}$.
But $\gamma$ is finite (in particular equidimensional), so this preimage still has codimension two in $V$.
Hence $\gamma^* \sigma$ is regular on all of $V$.
In other words,
\[ \gamma^* \sigma \in \Hnought V.\; \Sym^{[m]} \Omegarl V1{D_V}.. \]
Recall that the standard residue map commutes with \'etale pullback, and that $\gamma$ is \'etale over the general point of $P$.
So the two functions
\[ (\gamma|_{P_V})^* (\wt\sigma) \in \Hnought P_V.\O{P_V}(*\supp \gamma|_{P_V}^* {\rp P^c.}). \]
and
\[ \res_{P_V}^m (\gamma^* \sigma) \in \Hnought P_V.\O{P_V}. \]
agree on an open subset of $P_V$, hence everywhere.
This shows that $\wt\sigma$ is a regular function on $P_V$, as desired.

\subsubsection*{Step 2: Surjectivity}

It remains to show surjectivity of the maps $\res_P^m$.
This is a local question, so we may restrict ourselves to an open set $U \subset X$ admitting a map $\gamma \from V \to U$ as in \cref{dlt quotient}.
Let $G = \Gal(\gamma)$ be the Galois group of $\gamma$.
Start with the map
\[ \res_{P_V}^m \from \Sym^m \Omegal V1{D_V} \lto \O{P_V} \]
as before and note that we can also construct $\res_P^m$ by applying the functor $\gamma_*(-)^G$ to $\res_{P_V}^m$.
This means that we consider $U = \factor VG$ with the trivial $G$-action and look at the invariant sections of the relevant push-forward sheaves (which are $G$-sheaves in a natural way).
For more details, cf.~\cite[Appendix~A]{GKKP11}.

The claim now follows from the surjectivity of $\res_{P_V}^m$ (which is due to the fact that the pair $(V, D_V)$ is snc) and the exactness of the functor $\gamma_*(-)^G$.
This exactness holds because the order of $G$ is prime to~$p$ by the ``tamely dlt'' assumption, and therefore we have the usual Reynolds operator argument at our disposal.
Cf.~the characteristic zero version of this argument~\cite[Lemma~A.3]{GKKP11}.

\subsubsection*{Step 3: Residue sequence on $X$}

Next we prove the existence of sequence~\labelcref{res seq}.
The map $\res_P$ is of course nothing but the special case $m = 1$ of the maps just constructed.
By what we already know, we thus only need to show that its kernel is isomorphic to $\Omegarl X1{D-P}$.
But that kernel is a reflexive sheaf by~\cite[Cor.~1.5]{Hartshorne80}.
Furthermore it is isomorphic to $\Omegarl X1{D-P}$ on $\snc{(X, D)}$, by the usual residue sequence for snc pairs.
The isomorphism then extends to all of $X$ by reflexivity.

\subsubsection*{Step 4: Residue sequence on $P$}

Finally we turn to sequence~\labelcref{res P seq}.
Clearly, the reflexive restriction of $\res_P$ to $P$ is a surjective map $\res_P^P \from \Omegarl X1D \big|_P \ddual \lto \O P$, and it remains to show that its kernel is isomorphic to $\Omegal P1{\rd P^c.}$.
To this end, first note that there is a short exact sequence
\begin{equation} \label{ACD.1}
0 \lto \Omegarl X1D(-P) \ddual \lto \Omegarl X1D \lto \Omegarl X1D \big|_P \ddual \lto 0.
\end{equation}
In fact, the second map is surjective because on a regular curve, taking the double dual really just amounts to dividing out the torsion.
And by the same argument as in the previous step, the kernel is reflexive and thus isomorphic to $\Omegarl X1D(-P) \ddual$.

Consider now the commutative diagram with exact rows and columns depicted in Figure~\labelcref{857} \vpageref{857}.
\begin{figure}[!t]
\centerline{
\xymatrix@R=4ex{
& & 0 \ar[d] & 0 \ar@{..>}[d] \\
0 \ar[r] & \Omegarl X1D(-P) \ddual \ar[r] \ar@{=}[d] & \Omegarl X1{D-P} \ar^-{\restr_P}[r] \ar[d] & \Omegal P1{\rd P^c.} \ar@{..>}[d] \ar[r] & 0 \\
0 \ar[r] & \Omegarl X1D(-P) \ddual \ar[r] & \Omegarl X1D \ar[r] \ar^-{\res_P}[d] & \Omegarl X1D \big|_P \ddual \ar[r] \ar^-{\res_P^P}[d] & 0 \\
& & \O P \ar@{=}[r] \ar[d] & \O P \ar[d] \\
& & 0 & 0
} }
\caption{Diagram used in the proof of~\labelcref{res P seq}.} \label{857}
\end{figure}
The first row is the restriction sequence~\labelcref{restr seq}\footnote{Needless to say, the proof of \cref{restr} does not rely on \cref{res}---see \cref{restr proof} below.}, while the second row is~\labelcref{ACD.1}.
The middle column is~\labelcref{res seq}, the residue sequence on $X$.
The Snake Lemma then shows that the dotted arrow $\Omegal P1{\rd P^c.} \dashrightarrow \Omegarl X1D \big|_P \ddual$ exists, is injective, and that its image is exactly the kernel of $\res_P^P$.
The column on the right-hand side is therefore likewise exact, and it is precisely sequence~\labelcref{res P seq}. \qed

\subsection{Proof of \cref{restr}} \label{restr proof}

The proof of \cref{restr} is analogous to the proof of \cref{res}, hence we will only provide an outline, with most details omitted.
To begin with, if $(X, D)$ is snc then sequence~\labelcref{restr seq} reads
\[ 0 \lto \Omegal X1D(-P) \lto \Omegal X1{D-P} \xrightarrow{\;\restr_P\;} \Omegal P1{P^c} \lto 0 \]
and this exists by~\cite[2.3(c)]{EV92}.
In particular, we already have $\restr_P$ and its $m$-th symmetric power on the snc locus $\snc{(X, D)}$.
Pushing forward this symmetric power to all of $X$, we obtain a map
\begin{equation} \label{821}
\Sym^{[m]} \Omegarl X1{D-P} \lto \O P(mK_P)(* \rp P^c.),
\end{equation}
and we have to show that it factors via a map
\[ \Sym^{[m]} \Omegarl X1{D-P} \lto \O P(mK_P + \rd mP^c.), \]
for this will be the desired map $\restr_P^m$.
To this end, we have the following criterion:

\begin{clm} \label{993}
Notation as in the previous proof.
A local section $\wt\sigma$ of $\O P(mK_P)(* \rp P^c.)$ is contained in $\O P(mK_P + \rd mP^c.)$ if and only if $\gamma^*(\wt\sigma)$ is a regular section of $\O{P_V}(mK_{P_V} + mP_V^c)$, where $P_V^c \defn \Diff{P_V}{D_V - P_V} = (D_V - P_V)\big|_{P_V}$.
\end{clm}

\begin{proof}[Proof of \cref{993}]
The proof of this criterion is done by a local computation similar to the one in the proof of \cref{compute diff}, whose notation we adopt.
If locally at $x \in \supp \rp P^c.$, Case~\labelcref{dlt snc} holds, the claim is clear.
Therefore we focus on Case~\labelcref{dlt plt}.
Note that in this case $D_V$ is smooth, which implies $P_V = D_V$ and thus $P_V^c = 0$.

Write $\wt\sigma = t^k \, (\d t)^m$ with $k \in \Z$ (locally and up to units), and let $\ell$ be the Cartier index of $K_X + D$ at $x$.
The coefficient of $\rd mP^c.$ at $x$ is $\rd m (1 - \frac1\ell).$, therefore $\wt\sigma$ is contained in $\O P(mK_P + \rd mP^c.)$ if and only if
\begin{equation} \label{1005}
k \ge - \rd m (1 - \textstyle\frac1\ell)..
\end{equation}
On the other hand, $\gamma^* (t) = \eps u^\ell$ and hence
\[ \gamma^* (\wt\sigma) = \big( \eps^{k + m} \ell^m u^{k\ell + m(\ell - 1)} + \cdots \big) \cdot (\d u)^m, \]
where the dots stand for terms involving higher powers of $u$.
We see that $\gamma^* (\wt\sigma)$ is regular if and only if $k\ell + m(\ell - 1) \ge 0$.
This is equivalent to~\labelcref{1005}, proving the claim.
\end{proof}

Once the maps $\restr_P^m$ are constructed, their surjectivity follows from the right-exactness of $\gamma_*(-)^G$, as before.
Finally, to obtain sequence~\labelcref{restr seq} we set $\restr_P \defn \restr_P^1$.
On the snc locus $\snc{(X, D)}$, the kernel agrees with $\Omegarl X1D(-P) \ddual$ by the snc case mentioned in the beginning of the proof.
Since both sheaves are reflexive, they agree everywhere. \qed

\section{Lifting forms along a non-positive map} \label{sec lift elem}

The following theorem, while technical in nature, is at the heart of the paper.
The ``non-positivity'' in the title refers to property~\labelcref{lift elem.2} below.

\begin{thm}[Lifting forms] \label{lift elem}
Let $g \from Y \to X$ be a proper birational map of normal surfaces over a field $k$, with $E = \Exc(g)$ the reduced exceptional divisor.
Furthermore let $D$ be a reduced divisor on $X$, and set $D_Y \defn g\inv_* D + E$.
Assume the following:
\begin{enumerate}
\item\label{lift elem.1} The pair $(Y, D_Y)$ is tamely dlt, and
\item\label{lift elem.2} the anticanonical divisor $-(K_Y + D_Y)$ is $g$-nef.
\end{enumerate}
Then the natural map
\[ g_* \Omegarl{Y/k}1{D_Y} \xhookrightarrow{\quad} \Omegarl{X/k}1D \]
is an isomorphism.
\end{thm}

\subsection*{Step 0: Setup of notation and outline of proof strategy}

Let
\[ \sigma \in \HH0.X.\Omegarl X1D. \setminus \set0 \]
be a nonzero reflexive logarithmic $1$-form, and let $g^* \sigma$ be its pullback to $Y$, considered as a rational section of the sheaf $\Omegarl Y1{D_Y}$.
We want to show that $g^* \sigma$ is in fact a regular section of that sheaf.
To this end, first pick an effective $g$-exceptional divisor $G$ such that
\begin{equation} \label{rueckert}
g^* \sigma \in \HH0.Y.\Omegarl Y1{D_Y}(G) \ddual..
\end{equation}
For example, $G$ may be taken to be the pole divisor of the rational section $g^* \sigma$.
We will show that whenever $G$ is nonzero, there is a curve $P \subset \supp G$ such that~\labelcref{rueckert} continues to hold with $G$ replaced by $G - P$.
Iterating this argument finitely often, we arrive at $G = 0$, hence $g^* \sigma \in \HH0.Y.\Omegarl Y1{D_Y}.$ as desired.

\subsection*{Step 1: Residue sequence}

Assume that~\labelcref{rueckert} holds for some $G \ne 0$.
Then $G^2 < 0$ by the Negativity Lemma (applied on some resolution of $Y$) and consequently, $G \cdot P < 0$ for some exceptional curve $P \subset \supp G \subset E$.
Twisting by $\O Y(-G)$ and taking the reflexive hull, \labelcref{rueckert} induces a map $i \from \O Y(-G) \to \Omegarl Y1{D_Y}$.
As $g^* \sigma \ne 0$, this map is nonzero and hence injective.
On the tamely dlt pair $(Y, D_Y)$, we have the residue sequence~\labelcref{res seq}
\[ \xymatrix{
& & \O Y(-G) \ar@{ ir->}_-i[d] \ar@/^.5pc/[dr] \ar@/_.5pc/@{-->}_-j[dl] \\
0 \ar[r] & \Omegarl Y1{D_Y - P} \ar[r] & \Omegarl Y1{D_Y} \ar^-{\res_P}[r] & \O P \ar[r] & 0.
} \]

\begin{clm} \label{factor res}
The composition ${\res_P} \circ i$ is zero, and hence $i$ factors via a map $j$ as indicated by the dashed arrow in the above diagram.
\end{clm}

\begin{proof}[Proof of \cref{factor res}]
Let $m \ge 1$ be sufficiently divisible so that $mG$ is Cartier (recall that $Y$ is \Q-factorial).
The $m$-th reflexive symmetric power of $i$, composed with the map $\res_P^m$ from \cref{res}, yields a map
\begin{equation} \label{728}
\O Y(-mG) \xrightarrow{\Sym^{[m]} i} \Sym^{[m]} \Omegarl Y1{D_Y} \xrightarrow{\res_P^m} \O P
\end{equation}
which is nothing but the $m$-th reflexive symmetric power of ${\res_P} \circ i$.
Hence in order to show that ${\res_P} \circ i$ vanishes, it is sufficient to prove the vanishing of~\labelcref{728}.
As the target of the latter map is supported on $P$, it is zero if and only if its restriction to $P$ is zero.
But that restriction is a map $\O P(-mG) \to \O P$, or in other words, an element of $\HH0.P.\O P(mG).$.
As $G \cdot P < 0$ and $mG$ is Cartier, the latter space is zero.
\end{proof}

\subsection*{Step 2: Restriction sequence}

We essentially repeat Step~1, but with the residue sequence replaced by the restriction sequence~\labelcref{restr seq}:
\[ \xymatrix{
& & \O Y(-G) \ar@{ ir->}_-j[d] \ar@/^.5pc/[dr] \ar@/_.5pc/@{-->}_-\iota[dl] \\
0 \ar[r] & \Omegarl Y1{D_Y}(-P) \ddual \ar[r] & \Omegarl Y1{D_Y - P} \ar^-{\restr_P}[r] & \Omegal P1{\rd P^c.} \ar[r] & 0.
} \]

\begin{clm} \label{factor restr}
The composition ${\restr_P} \circ j$ is zero, and hence $j$ factors via a map $\iota$ as indicated by the dashed arrow in the above diagram.
\end{clm}

\begin{proof}[Proof of \cref{factor restr}]
Let $m$ be as in the proof of \cref{factor res}, so that $mG$ is Cartier.
The $m$-th reflexive symmetric power of $j$, composed with the map $\restr_P^m$ from \cref{restr}, is the $m$-th reflexive symmetric power of ${\restr_P} \circ j$:
\begin{equation} \label{751}
\O Y(-mG) \xrightarrow{\Sym^{[m]} j} \Sym^{[m]} \Omegarl Y1{D_Y - P} \xrightarrow{\restr_P^m} \O P(mK_P + \rd mP^c.)
\end{equation}
As in \cref{factor res}, it suffices to show that the restriction of~\labelcref{751} to $P$ vanishes.
This is a map $\O P(-mG) \to \O P(mK_P + \rd mP^c.)$, or in other words, an element of $\HH0.P.\O P(mK_P + {\rd mP^c.} + mG).$.
As
\begin{align*}
\deg \big( mK_P + \rd mP^c. + mG\big|_P \big) & \le \deg \big( m(K_P + P^c) + mG\big|_P \big) && \text{round-down} \\
& \le \deg \big( m(K_Y + D_Y + G)\big|_P \big) && \text{by adjunction} \\
& \le \deg \big( mG\big|_P \big) && \text{by~\labelcref{lift elem.2}} \\*
& = mG \cdot P && \text{$mG$ is Cartier} \\*
& < 0,
\end{align*}
the latter space is zero.
This ends the argument.
\end{proof}

The proof of \cref{lift elem} is now easily finished: the existence of the map $\iota$ is equivalent to giving a global section of the sheaf $\Omegarl Y1{D_Y}(G - P) \ddual$, which of course is exactly the form $g^* \sigma$ we started with.
This shows that~\labelcref{rueckert} holds with $G - P$ in place of $G$, as desired. \qed

\section{Proof of \cref{main lc}} \label{sec main lc}

The aim of this section is to prove our first main result: any log canonical surface pair $(X, D)$ over a perfect field of characteristic $\ge 7$ satisfies the \lext.
The following notion will play a key role.

\begin{dfn}[Tame resolutions] \label{tame res}
Let $(X, D)$ be a reduced log canonical surface pair over a field $k$.
A \emph{tame resolution} of $(X, D)$ is a log resolution $\pi \from Y \to X$ together with a factorization of $\pi$ as in \cref{fact res} such that
\begin{enumerate}
\item\label{tame.1} for any $0 \le i \le r - 1$, the pair $(Y_i, \wt D_i)$ is tamely dlt, and
\item\label{tame.2} if $f$ is not an isomorphism (this can happen only if $(X, D)$ is not plt), then also $(Z, \wt D_r)$ is required to be tamely dlt.
\end{enumerate}
\end{dfn}

\subsection{Auxiliary results}

First we show that when dealing with log canonical surface pairs, there is no loss of generality in assuming them to be reduced.
We also prove that having a tame resolution implies the \lext and that the \lext is invariant under separable base change.
The latter property is used for reducing to the case of an algebraically closed ground field, where the classification of surface singularities becomes simpler.

\begin{prp}[Rounding down] \label{rounddown}
Let $(X, D)$ be a log canonical surface pair.
Then also $(X, \rd D.)$ is log canonical.
\end{prp}

\begin{prp}[Tameness is sufficient] \label{tame lext}
Let $(X, D)$ be a reduced log canonical surface pair admitting a tame resolution.
Then $(X, D)$ satisfies the \lext for $1$-forms.
\end{prp}

\begin{prp}[Base change] \label{base change}
Let $(X, D)$ be a pair defined over a field $k$, and consider a separable field extension $k'/k$.
Set $X' \defn X \x_k k'$ and $D' \defn D \x_k k'$.
\begin{enumerate}
\item\label{base change.1} If $(X', D')$ satisfies the \et for $q$-forms, for some value of $q$, then so does $(X, D)$.
\item\label{base change.2} If $(X, D)$ admits a log resolution, the converse of~\labelcref{base change.1} also holds.
\end{enumerate}
Ditto for the \lext.
\end{prp}

\begin{proof}[Proof of \cref{rounddown}]
If $(X, D)$ is numerically log canonical, then so is $(X, \rd D.)$.
Thus it suffices to show that $K_X + \rd D.$ is \Q-Cartier.
The question is local, so we may concentrate attention on a point $x \in \supp {\{ D \}}$, the fractional part of $D$.
At such a point, the pair $(X, \rd D.)$ is even numerically dlt and then $(X, \emptyset)$ is numerically klt.
Applying \cref{fact res} to the latter pair, we get that $f \from Z \to X$ is an isomorphism, as there are no exceptional divisors of discrepancy $-1$, and hence $X$ is even \Q-factorial because $Z$ is.

An alternative (yet closely related) argument goes by noting that the characteristic zero proof of~\cite[Prop.~4.11]{KM98} still works if we replace the use of the basepoint-free theorem~\cite[Thm.~3.3]{KM98} by~\cite[Thm.~4.2]{Tanaka18}.
\end{proof}

\begin{proof}[Proof of \cref{tame lext}]
Let $\pi \from Y \to X$ be a tame resolution of $(X, D)$, where we keep notation from \cref{fact res}.
It suffices to extend $1$-forms along each step of the given factorization separately.
That is, we will prove the following two statements:

\begin{clm} \label{lift f}
The sheaf $f_* \Omegarl Z1{\wt D_r}$ is reflexive.
\end{clm}

\begin{clm} \label{lift phi}
For any $0 \le i \le r - 1$, the sheaf $(\phi_i)_* \, \Omegarl{Y_i}1{\wt D_i}$ is reflexive.
\end{clm}

\begin{proof}[Proof of \cref{lift f}]
If $(X, D)$ is plt, then $f$ is an isomorphism and there is nothing to prove.
Otherwise, we would like to apply \cref{lift elem}.
The tameness condition~\labelcref{lift elem.1} is satisfied by~\labelcref{tame.2}.
It remains to check~\labelcref{lift elem.2}, i.e.~that $-(K_Z + \wt D_r)$ is $f$-nef.
To this end, let $P \subset Z$ be any $f$-exceptional curve and note that
\begin{align*}
\big( K_Z + \wt D_r \big) \cdot P & = f^* (K_X + D) \cdot P && \text{by~\labelcref{fact res.3}} \\
& = 0 && \text{as $f_* P = 0$.}
\end{align*}
So $K_Z + \wt D_r$ is even $f$-numerically trivial.
\cref{lift f} is proved.
\end{proof}

\begin{proof}[Proof of \cref{lift phi}]
Again, we will apply \cref{lift elem} and only Condition~\labelcref{lift elem.2}, the $\phi_i$-nefness of $-(K_{Y_i} + \wt D_i)$, needs to be checked.
Let $P \subset Y_i$ be the unique $\phi_i$-exceptional curve.
Since $(Y_{i+1}, \wt D_{i+1})$ is dlt (in particular, log canonical) and $\wt D_i = (\phi_i\inv)_* \big( \wt D_{i+1} \big) + P$, we have
\begin{equation} \label{1144}
K_{Y_i} + \wt D_i = \phi_i^* \big( K_{Y_{i+1}} + \wt D_{i+1} \big) + \lambda P,
\end{equation}
where $\lambda = a \big( P, Y_{i+1}, \wt D_{i+1} \big) + 1 \ge 0$.
On the other hand, $P^2 < 0$ by the Negativity Lemma.
Hence
\begin{align*}
(K_{Y_i} + \wt D_i) \cdot P & = \big( \phi_i^* \big( K_{Y_{i+1}} + \wt D_{i+1} \big) + \lambda P \big) \cdot P && \text{by~\labelcref{1144}} \\
& = \lambda \cdot P^2 && \text{as $(\phi_i)_* P = 0$} \\
& \le 0.
\end{align*}
\cref{lift phi} now follows from \cref{lift elem}.
\end{proof}

By \cref{lift f} and \cref{lift phi}, also the sheaf
\begin{align*}
\pi_* \Omegal Y1{\pi\inv_* D + E} & = (f \circ \phi_{r-1} \circ \dots \circ \phi_0)_* \, \Omegal{Y_0}1{\wt D_0} \\
& = (f \circ \phi_{r-1} \circ \dots \circ \phi_1)_* \, \Omegarl{Y_1}1{\wt D_1} \\
& = \cdots \\
& = f_* \Omegarl{Y_r}1{\wt D_r} \\*
& = \Omegarl X1D
\end{align*}
is reflexive.
The proof of \cref{tame lext} is thus finished.
\end{proof}

\begin{proof}[Proof of \cref{base change}]
For any object (variety, map, sheaf, \dots) over $k$, we denote the base change to $k'$ by $(-)'$.
Concerning~\labelcref{base change.1}, let $\pi \from Y \to X$ be proper birational, with $Y$ normal.
Then there is a commutative diagram
\[ \xymatrix{
Y' \ar[rr] \ar_-{\pi'}[d] & & Y \ar^-\pi[d] \\
X' \ar[rr] & & X \\
} \]
Since $k'/k$ is a separable field extension, the horizontal maps are \'etale and faithfully flat.
In particular, $X'$ and $Y'$ are still normal, and after possibly replacing them by suitable connected components, $\pi'$ is proper birational.
By assumption, $\pi'_* \Omegar{Y'/k'}q$ is reflexive.
But
\begin{align*}
\pi'_* \Omegar{Y'/k'}q & = \pi'_* \left( \big( \Omegap{Y'/k'}q \big) \ddual \right) & & \text{by definition} \\
& = \pi'_* \bigg[ \Big( \big( \Omegap{Y/k}q \big)' \Big) \ddual \bigg] & & \text{\cite[Ch.~II, Prop.~8.10]{Har77}} \\
& = \pi'_* \left( \big( \Omegar{Y/k}q \big)' \right) & & \PreprintAndPublication{\text{\labelcref{flat refl.1}}}{} \\
& = \big( \pi_* \Omegar{Y/k}q \big)' & & \text{\cite[Ch.~III, Prop.~9.3]{Har77},}
\end{align*}
hence the claim follows\PreprintAndPublication{ from~\labelcref{flat refl.3}}{}.

For~\labelcref{base change.2}, keep notation but assume additionally that $\pi$ is a resolution of singularities.
By the above argument\PreprintAndPublication{ and~\labelcref{flat refl.2}}{}, the sheaf $\pi'_* \Omegar{Y'/k'}q$ is reflexive.
Because $Y' \to Y$ is \'etale, $\pi'$ is in fact a resolution and it follows that $X'$ satisfies the \et for $q$-forms.

The proof in the logarithmic case is similar, and therefore omitted.
\end{proof}

\PreprintAndPublication{
\begin{lem}[Dual commutes with base change] \label{flat refl}
Let $R$ be a noetherian ring, $M$ a finitely generated $R$-module, and $R \subset S$ a flat ring extension.
Set $M_S \defn M \tensor_R S$.
Then:
\begin{enumerate}
\item\label{flat refl.1} The natural map $\Hom R.M.R. \tensor_R S \to \Hom S.M_S.S.$ is an isomorphism.
\item\label{flat refl.2} If $M$ is reflexive, then so is $M_S$.
\item\label{flat refl.3} If $R \subset S$ is faithfully flat, the converse of~\labelcref{flat refl.2} also holds.
\end{enumerate}
\end{lem}

\begin{proof}
Let $F_1 \to F_0 \to M$ be a finite presentation of $M$.
Dualizing and tensorizing with $S$, we get the first row in the following commutative diagram.
First tensorizing and then dualizing gives the second row.
\[ \xymatrix{
0 \ar[r] & \Hom R.M.R. \tensor_R S \ar[r] \ar[d] & \Hom R.F_0.R. \tensor_R S \ar[r] \ar@{=}[d] & \Hom R.F_1.R. \tensor_R S \ar@{=}[d] \\
0 \ar[r] & \Hom S.M_S.S. \ar[r] & \Hom S.F_{0,S}.S. \ar[r] & \Hom S.F_{1,S}.S.
} \]
The first row is exact because $S$ is flat over $R$, and the second row is exact for general reasons.
The leftmost vertical arrow is the map in question, while the other two are isomorphisms because the $F_i$ are free.
\labelcref{flat refl.1} now follows from the Snake Lemma.

For~\labelcref{flat refl.2}, consider the natural isomorphism $M \to M \ddual$.
By~\labelcref{flat refl.1}, after tensorizing with $S$ it becomes the natural map $M_S \to M_S \ddual$, and it obviously stays an isomorphism.
Hence $M_S$ is reflexive, too.
If $S$ is faithfully flat, we may run the argument backwards, proving~\labelcref{flat refl.3}.
\end{proof}
}{}

\subsection{Proof of \cref{main lc}}

By \cref{rounddown}, we may assume that $(X, D)$ is reduced.
Furthermore, since our ground field $k$ is assumed to be perfect, its algebraic closure $\bar k$ is separable over $k$ and hence by \cref{base change}, we may assume that $k = \bar k$.
The singularities of reduced log canonical surface pairs over an algebraically closed ground field have been classified in~\cite[Cor.~3.31, 3.39, 3.40]{Kollar13}.
According to this classification, there are seven cases to be considered.
Their dual graphs are depicted in Figures~\labelcref{fig cusp}--\labelcref{fig other quot} \vpageref{fig cusp} (the first case is not shown since it has only one exceptional curve).
Here we use the following color and labeling pattern.
The extra information thus contained in the figures is easily verified.

\begin{ntn}
A \emph{plain circle} denotes an exceptional curve with discrepancy equal to~$-1$.
A node shaded in \emph{gray} denotes an exceptional curve with discrepancy $> -1$.
All exceptional curves are smooth rational.
The components of $\pi\inv_* D$ are shown in~\emph{black}.
A \emph{negative number} attached to a vertex denotes the self-intersection of the corresponding curve.
A \emph{leaf} is a curve intersecting at most one other curve, while a \emph{fork} intersects at least three other curves.
\end{ntn}

\begin{figure}[!ht]
\centering
\begin{tikzpicture}
\node[circle,draw] (node0) at (3, 0) { };
\node[circle,draw] (node1) at (2, 1) { };
\node[circle,draw] (node2) at (1, 1) { };
\node (node3) at (0, 1) { \dots };
\node[circle,draw] (node4) at (-1, 1) { };
\node[circle,draw] (node5) at (-2, 1) { };
\node[circle,draw] (node6) at (-3, 0) { };
\node[circle,draw] (node7) at (-2, -1) { };
\node[circle,draw] (node8) at (-1, -1) { };
\node (node9) at (0, -1) { \dots };
\node[circle,draw] (node10) at (1, -1) { };
\node[circle,draw] (node11) at (2, -1) { };

\path (node0) edge[bend right=40] node {} (node1);
\path (node1) edge node {} (node2);
\path (node2) edge node {} (node3);
\path (node3) edge node {} (node4);
\path (node4) edge node {} (node5);
\path (node5) edge[bend right=40] node {} (node6);
\path (node6) edge[bend right=40] node {} (node7);
\path (node7) edge node {} (node8);
\path (node8) edge node {} (node9);
\path (node9) edge node {} (node10);
\path (node10) edge node {} (node11);
\path (node11) edge[bend right=40] node {} (node0);
\end{tikzpicture}
\caption{Cusp~\cite[(3.39.2)]{Kollar13}.} \label{fig cusp}

\vspace*{\floatsep}

\begin{tikzpicture}
\node[circle,draw,fill=black!15] (leaf0) at (3, 1) [label=right:$-2$] { };
\node[circle,draw,fill=black!15] (leaf1) at (-3, 1) [label=left:$-2$] { };
\node[circle,draw,fill=black!15] (leaf2) at (-3, -1) [label=left:$-2$] { };
\node[circle,draw,fill=black!15] (leaf3) at (3, -1) [label=right:$-2$] { };

\node[circle,draw] (fork0) at (-2, 0) { };
\node[circle,draw] (fork1) at (2, 0) { };
\node[circle,draw] (node0) at (-1, 0) { };
\node[circle,draw] (node1) at (1, 0) { };
\node (dotsnode) at (0, 0) { \dots };

\path (leaf0) edge node {} (fork1);
\path (leaf1) edge node {} (fork0);
\path (leaf2) edge node {} (fork0);
\path (leaf3) edge node {} (fork1);
\path (fork0) edge node {} (node0);
\path (node0) edge node {} (dotsnode);
\path (dotsnode) edge node {} (node1);
\path (node1) edge node {} (fork1);
\end{tikzpicture}
\caption{$\Z/2$-quotient of cusp or simple elliptic~\cite[(3.39.3)]{Kollar13}.} \label{fig z2 cusp}

\vspace*{\floatsep}

\begin{tikzpicture}
\node[circle,draw] (fork) at (0, 0) { };
\node[circle,draw,fill=black!15] (node0) at (-3, 0) { };
\node (node1) at (-2, 0) { \dots };
\node[circle,draw,fill=black!15] (node2) at (-1, 0) { };
\node[circle,draw,fill=black!15] (node3) at (1, 0) { };
\node (node4) at (2, 0) { \dots };
\node[circle,draw,fill=black!15] (node5) at (3, 0) { };
\node[draw,fill=black!15] (node6) at (0, -1.25) { $\Gamma_3$};

\path (node0) edge node {} (node1);
\path (node1) edge node {} (node2);
\path (node2) edge node {} (fork);
\path (fork) edge node {} (node3);
\path (node3) edge node {} (node4);
\path (node4) edge node {} (node5);
\path (fork) edge node {} (node6);

\draw[thick,decoration={brace,mirror,amplitude=8},decorate] ([yshift=-3mm]node0.west) --node[below=3mm]{$\Gamma_1$} ([yshift=-3mm]node2.east);
\draw[thick,decoration={brace,mirror,amplitude=8},decorate] ([yshift=-3mm]node3.west) --node[below=3mm]{$\Gamma_2$} ([yshift=-3mm]node5.east);

\end{tikzpicture}
\caption{Other quotient of simple elliptic~\cite[(3.39.4)]{Kollar13}.
$\Gamma_3$~is likewise a chain and $(\det \Gamma_i)$ is either $(3, 3, 3)$, $(2, 4, 4)$ or $(2, 3, 6)$.} \label{fig other quot simple elliptic}

\vspace*{\floatsep}

\begin{tikzpicture}
\node[circle,draw,fill=black!15] (node00) at (-2, 0) { };
\node[circle,draw,fill=black!15] (node01) at (2, 0) { };
\node[circle,draw,fill=black!15] (node02) at (-1, 0) { };
\node[circle,draw,fill=black!15] (node03) at (1, 0) { };
\node (dotsnode0) at (0, 0) { \dots };
\node[circle,draw,fill=black] (boundary00) at (-3, 0) { };

\path (node00) edge node {} (node02);
\path (node02) edge node {} (dotsnode0);
\path (dotsnode0) edge node {} (node03);
\path (node03) edge node {} (node01);
\path (node00) edge node {} (boundary00);

\node[circle,draw,fill=black!15] (node10) at (-2, 1) { };
\node[circle,draw,fill=black!15] (node11) at (2, 1) { };
\node[circle,draw,fill=black!15] (node12) at (-1, 1) { };
\node[circle,draw,fill=black!15] (node13) at (1, 1) { };
\node (dotsnode1) at (0, 1) { \dots };

\path (node10) edge node {} (node12);
\path (node12) edge node {} (dotsnode1);
\path (dotsnode1) edge node {} (node13);
\path (node13) edge node {} (node11);

\node[circle,draw] (node20) at (-2, -1) { };
\node[circle,draw] (node21) at (2, -1) { };
\node[circle,draw] (node22) at (-1, -1) { };
\node[circle,draw] (node23) at (1, -1) { };
\node (dotsnode2) at (0, -1) { \dots };
\node[circle,draw,fill=black] (boundary20) at (-3, -1) { };
\node[circle,draw,fill=black] (boundary21) at (3, -1) { };

\path (node20) edge node {} (node22);
\path (node22) edge node {} (dotsnode2);
\path (dotsnode2) edge node {} (node23);
\path (node23) edge node {} (node21);
\path (node20) edge node {} (boundary20);
\path (node21) edge node {} (boundary21);

\node at (0, .5) { or };
\node at (0, -.5) { or };
\end{tikzpicture}
\caption{The three possibilities for a cyclic quotient~\cite[(3.40.1)]{Kollar13}.} \label{fig cyclic}

\vspace*{\floatsep}

\begin{tikzpicture}
\node[circle,draw,fill=black!15] (leaf0) at (3, 1) [label=right:$-2$] { };
\node[circle,draw,fill=black!15] (leaf3) at (3, -1) [label=right:$-2$] { };
\node[circle,draw,fill=black!15] (fork0) at (2, 0) { };
\node[circle,draw,fill=black!15] (node0) at (0, 0) { };
\node (dotsnode) at (1, 0) { \dots };

\path (leaf0) edge node {} (fork0);
\path (leaf3) edge node {} (fork0);
\path (fork0) edge node {} (dotsnode);
\path (dotsnode) edge node {} (node0);

\node[circle,draw,fill=black!15] (leaf10) at (10, 1) [label=right:$-2$] { };
\node[circle,draw,fill=black!15] (leaf11) at (10, -1) [label=right:$-2$] { };
\node[circle,draw] (fork1) at (9, 0) { };
\node[circle,draw] (node10) at (7, 0) { };
\node (dotsnode1) at (8, 0) { \dots };
\node[circle,draw,fill=black] (boundary) at (6, 0) { };

\path (leaf10) edge node {} (fork1);
\path (leaf11) edge node {} (fork1);
\path (fork1) edge node {} (dotsnode1);
\path (dotsnode1) edge node {} (node10);
\path (node10) edge node {} (boundary);

\node at (4.75, 0) { or };
\end{tikzpicture}
\caption{The two possibilities for a dihedral quotient~\cite[(3.40.2)]{Kollar13}.} \label{fig dihedral}

\vspace*{\floatsep}

\begin{tikzpicture}
\node[circle,draw,fill=black!15] (fork) at (0, 0) { };
\node[circle,draw,fill=black!15] (node0) at (-3, 0) { };
\node (node1) at (-2, 0) { \dots };
\node[circle,draw,fill=black!15] (node2) at (-1, 0) { };
\node[circle,draw,fill=black!15] (node3) at (1, 0) { };
\node (node4) at (2, 0) { \dots };
\node[circle,draw,fill=black!15] (node5) at (3, 0) { };
\node[draw,fill=black!15] (node6) at (0, -1.25) { $\Gamma_3$};

\path (node0) edge node {} (node1);
\path (node1) edge node {} (node2);
\path (node2) edge node {} (fork);
\path (fork) edge node {} (node3);
\path (node3) edge node {} (node4);
\path (node4) edge node {} (node5);
\path (fork) edge node {} (node6);

\draw[thick,decoration={brace,mirror,amplitude=8},decorate] ([yshift=-3mm]node0.west) --node[below=3mm]{$\Gamma_1$} ([yshift=-3mm]node2.east);
\draw[thick,decoration={brace,mirror,amplitude=8},decorate] ([yshift=-3mm]node3.west) --node[below=3mm]{$\Gamma_2$} ([yshift=-3mm]node5.east);

\end{tikzpicture}
\caption{Other quotient of a smooth surface~\cite[(3.40.3)]{Kollar13}.
$(\det \Gamma_i)$ is either $(2, 3, 3)$, $(2, 3, 4)$ or $(2, 3, 5)$.} \label{fig other quot}
\end{figure}

Since \cref{main lc} is local, we may shrink $X$ and assume that $(X, D)$ has only one singular point.
We use notation from \cref{fact res}, applied to the minimal resolution $\pi \from Y \to X$ of $(X, D)$.
In particular, $E$ is the exceptional locus of $\pi$ and $r$ is the number of contractions performed by the MMP before the minimal dlt model is reached.
The classification is then as follows.
(The names are actually valid only in char\-acteristic zero.
Here they are only meant for easier reference and should not be taken literally.)

\begin{enumerate}
\stepcounter{thm}
\item (Simple elliptic,~\cite[(3.39.1)]{Kollar13}) Here $D = 0$ and $E$ consists of a single smooth elliptic curve, which has discrepancy $-1$.
So $r = 0$ and the tameness condition on $\pi$ is automatically satisfied.
In this case, \cref{main lc} thus follows directly from \cref{tame lext}.

\item (Cusp, \cref{fig cusp}) Again, there are no curves of discrepancy $> -1$, so $r = 0$ and we conclude as before.

\item\label{z2 cusp} ($\Z/2$-quotient of cusp or simple elliptic, \cref{fig z2 cusp}) Here $r = 4$ and each step $\phi_i$ contracts a curve $C_i \subset U_i \subset Y_i$, where $U_i$ is a smooth open subset of~$Y_i$ and $C_i^2 = -2$.
By~\cite[Thm.~3.32]{Kollar13}, the resulting singularity is \'etale locally\footnote{As stated, \cite[Thm.~3.32]{Kollar13} gives the result only up to completion (which would also be sufficient), but the proof shows that there is a map $\phi_i(U_i) \to \Spec k[u^2, uv, v^2]$ which is \'etale at the point~$\phi_i(C_i)$.} isomorphic to the vertex of $\Spec k[u^2, uv, v^2]$.
Since $\operatorname{char} k \ge 7 > 2$, this singularity is actually a $\Z/2$-quotient of a smooth surface.
Then by the usual norm argument, $2 \cdot \Delta_i$ is Cartier for every integral Weil divisor $\Delta_i$ on $Y_i$, where $0 \le i \le 4$.
Applying this to $K_{Y_i} + \wt D_i$, we see that $\pi$ is a tame resolution and we conclude by \cref{tame lext}.

\item\label{other quot simple elliptic} (Other quotient of simple elliptic, \cref{fig other quot simple elliptic}) The three chains of rational curves $\Gamma_i$ can obviously be treated independently of each other, hence we will concentrate on, say, $\Gamma_1$.
The first curve contracted has to be the leaf, since otherwise there would be two components of $\wt D_1$ intersecting at a singular point of $Y_1$, contradicting \cref{dlt structure}.
Repeating this argument, we see that the curves in $\Gamma_1$ are contracted in sequence, starting from the leaf and proceeding towards the fork.
In particular, at each step there is only one singular point and it is obtained by contracting a subchain of $\Gamma_1$.
But $\det \Gamma_1 \le 6$ and by the recurrence relation in~\cite[(3.33.1)]{Kollar13}, the same also holds for all its subchains.
By~\cite[Thm.~3.32]{Kollar13} and the assumption $\operatorname{char} k \ge 7$, the singular point of each $Y_i$ is a quotient by a finite group of order $\le 6$.
As in the previous case~\labelcref{z2 cusp}, this implies that $\pi$ is tame and hence \cref{main lc} holds also in this case.

\item\label{cyclic} (Cyclic quotient, \cref{fig cyclic}) There are three subcases, according to whether the boundary $D$ has zero, one or two components.
In the first two cases, it actually not true that $(X, D)$ has a tame resolution, since the chain $E$ can be arbitrarily long and hence infinitely many (in fact, all) primes would have to be excluded.
So we cannot apply \cref{tame lext}.
But note that for every exceptional curve $P \subset Y = Y_0$, we have $\deg ( K_P + P^c) \le 0$, where $P^c \defn \Diff P{\wt D_0 - P}$.
(The degree is $-1$ for the leaves and $0$ for the other curves, since there is no fork.)
Also the pair $(Y_0, \wt D_0)$ is clearly tamely dlt, since it is even snc.
Hence in these cases, \cref{main lc} is a direct consequence of \cref{lift elem} applied to~$\pi$.
In the third subcase, we may follow the same argument or else note that $r = 0$, so $\pi$ is tame---it boils down to the same thing.

\item (Dihedral quotient, \cref{fig dihedral}) We have two subcases: either $D = 0$ or $D \ne 0$.
If $D = 0$, again there may not be a tame resolution.
Instead, the MMP needs to be chosen in such a way that first the two $(-2)$-curves intersecting the fork are contracted.
The resulting pairs $(Y_i, \wt D_i)$, $i = 1, 2$, are tamely dlt by the same reasoning as in case~\labelcref{z2 cusp}.
If $P \subset Y_2$ is the image of the fork, both singular points of $Y_2$ appear in $P^c \defn \Diff P{\wt D_2 - P}$ with coefficient either\footnote{As $K_{Y_2} + \wt D_2$ is not Cartier, the coefficient is actually $1/2$, but we only need an upper bound on the different.} zero or $\frac12$, by \cref{compute diff}.
Hence
\[ \deg ( K_P + P^c ) \le -2 + 1 + \frac12 + \frac12 = 0. \]
If $P \subset Y_2$ is any other exceptional curve, the above inequality also holds, as in case~\labelcref{cyclic}.
We can therefore apply \cref{lift elem} to the map $Y_2 \to X$ to conclude.

If $D \ne 0$, then $r = 2$ and only the two $(-2)$-curves are contracted.
The resolution $\pi$ is then tame by exactly the same argument as in case~\labelcref{z2 cusp}.

\item (Other quotient of a smooth surface, \cref{fig other quot}) The argument is similar to case~\labelcref{other quot simple elliptic}.
First the chains $\Gamma_i$ are contracted, starting from the leaves and progressing towards the fork.
As $\det \Gamma_i \le 5 < 7 \le \operatorname{char} k$, this implies that $(Y_i, \wt D_i)$ is tamely dlt for $0 \le i \le r - 1$.
Furthermore $X$ is log terminal and $D = 0$, so $(X, D)$ is plt and case~\labelcref{tame.2} of the definition of tameness applies.
So $\pi$ is tame and \cref{tame lext} gives the result.
\end{enumerate}

\noindent
Since we have now worked our way through all the cases, the proof of \cref{main lc} is finished. \qed

\section{Proof of \cref{main reg}} \label{sec main reg}

This section contains the proof of our second main result, \cref{main reg}.
The argument proceeds in three steps.

\subsection{Passing to a log resolution}

First of all, by blowing up $Y$ further we may turn $\Exc(\pi) + D_Y$ into an snc divisor.
We need to show that this does not change $\det(E_i \cdot E_j)$ up to sign.
Indeed, after renumbering we may assume that we are blowing up a point $p \in Y$ which is contained exactly in the exceptional curves $E_1, \dots, E_k$, where $k \le \ell$.
Let $r_i$ be the multiplicity of $E_i$ at $p$.
Set $A = (a_{ij}) = ( - E_i \cdot E_j)$, the negative of the intersection matrix on $Y$, and $\wt A$ the analogous matrix after blowing up $p$, with the new exceptional curve put first.
Also, let $A_0 = \sMat{0 & 0 \\ 0 & A}$ with one additional zero row and column.
Then
\[ \wt A = A_0 + \Mat{
1 & - r_1 & - r_2 & \cdots & - r_k & 0 \\
- r_1 & r_1^2 & r_1 r_2 & \cdots & r_1 r_k & 0 \\
\vdots & \vdots & \vdots & \ddots & \vdots & \vdots \\
- r_k & r_1 r_k & r_2 r_k & \cdots & r_k^2 & 0 \\
0 & 0 & 0 & \dots & 0 & 0
}, \]
where there are $\ell - k$ zero rows and columns, respectively.
By adding $r_i \cdot \text{(first column)}$ to the $(i + 1)$-st column for all $1 \le i \le k$, the matrix $\wt A$ is transformed into
\[ \wt A' = \Mat{
1 & 0 & \cdots & 0 \\
- r_1 & a_{11} & \cdots & a_{1\ell} \\
\vdots & \vdots & & \vdots \\
- r_k & \vdots & & \vdots \\
0 & a_{\ell 1} & \dots & a_{\ell \ell}
}, \]
while keeping the determinant unchanged.
Expanding $\det \wt A'$ by the first row, we see that $\det \wt A' = \det A$.
Hence we may make the following

\begin{awlog}
The map $\pi \from Y \to X$ is a log resolution of $(X, D)$.
\end{awlog}

\subsection{Dropping the non-exceptional divisor} \label{AB 1}

Pick an irreducible component $P \subset \rd D.$, and let $P_Y$ be its strict transform on $Y$.
Then consider the short exact sequence given by the residue map~\cite[2.3(b)]{EV92}
\[ 0 \lto \Omegal Y1{\rd D_Y. + E - P_Y} \lto \Omegal Y1{\rd D_Y. + E} \lto \O{P_Y} \lto 0. \]
Pushing it forward via $\pi$ yields
\[ 0 \lto \pi_* \Omegal Y1{\rd D_Y. + E - P_Y} \lto \pi_* \Omegal Y1{\rd D_Y. + E} \lto \sQ \lto 0, \]
where $\sQ \subset \pi_* \O{P_Y}$.
In particular, $\sQ$ is supported on $P$ and it is torsion-free as an $\O P$-module.
Hence $\sQ$ has only one associated prime, which is of height $1$.
It then follows from~\cite[Cor.~1.5]{Hartshorne80} that the sheaf $\pi_* \Omegal Y1{\rd D_Y. + E - P_Y}$ is likewise reflexive.
Repeating this argument for all components $P \subset \rd D.$, we arrive at the conclusion that
\[ \pi_* \Omegal Y1E \]
is reflexive, and hence isomorphic to $\Omegar X1$.
In other words, $(X, \emptyset)$ satisfies the \lext.

\subsection{Dropping the exceptional divisor}

Set $E = E_1 + \cdots + E_\ell$, and consider the residue sequence~\cite[2.3(a)]{EV92}
\[ 0 \lto \Omegap Y1 \lto \Omegal Y1E \lto \bigoplus_{i=1}^\ell \O{E_i} \lto 0. \]
We need to show that
\[ \HH0.Y.\Omegap Y1. \lto \HH0.Y.\Omegal Y1E. \]
is an isomorphism.
It suffices to show that the connecting homomorphism
\[ \delta \from \bigoplus_{i=1}^\ell \HH0.E_i.\O{E_i}. \lto \HH1.Y.\Omegap Y1. \]
is injective.
To this end, consider the restriction map
\[ r \from \HH1.Y.\Omegap Y1. \lto \bigoplus_{i=1}^\ell \HH1.E_i.\Omegap{E_i}1.. \]
We will show that the composition
\begin{equation} \label{r circ delta}
r \circ \delta \from \bigoplus_{i=1}^\ell \HH0.E_i.\O{E_i}. \lto \bigoplus_{i=1}^\ell \HH1.E_i.\Omegap{E_i}1.
\end{equation}
is an isomorphism.
In fact, on the left-hand side choose the basis consisting of the constant functions $\mathbf 1_{E_i}$, and on each summand of the right-hand side, choose the basis canonically determined by the trace map.
It is easy to see\footnote{For more details, the reader is advised to consult the proof of~\cite[Prop.~3.2]{GK13}, which is independent of the characteristic.} that with respect to these bases, \labelcref{r circ delta} is given by the intersection matrix $A \defn (E_i \cdot E_j)$.
By the Negativity Lemma~\cite[Lemma~3.40]{KM98}, $A$ is negative definite (in particular, invertible) when considered as an integer matrix.
Here, of course, we have to regard $A$ as defined over our ground field $k$ instead.
However, by our assumption, the characteristic $p$ of $k$ does not divide $\det A$.
Hence the matrix $A$ remains invertible when reduced modulo $p$.
In other words, $r \circ \delta$ is an isomorphism, and then $\delta$ is injective.
It follows that the sheaf $\pi_* \Omegap Y1$ is reflexive. \qed

\section{The characteristic zero Extension Theorem revisited} \label{ext higher-dim zero}

The purpose of this section is to explain how the ideas in this paper yield a new proof of~\cite[Thm.~1.5]{GKKP11}, repeated below as \cref A.
Even though we are ultimately only interested in that statement, in order to give a self-contained argument we have to set up an inductive procedure involving \cref B below.
The latter statement has already been proven in much greater generality in~\cite[Thm.~1.2]{Gra12}, but we must not use that result in our proof in order to avoid a circular dependence on~\cite{GKKP11}.

For a very brief run-down of \cC-pairs and \cC-differentials, we refer to \cref{sec C-pairs} below.
In the whole section, all varieties are assumed to be defined over the complex numbers.

\begin{thmalph} \label{A}
Let $(X, D)$ be a complex log canonical pair.
Then $(X, D)$ satisfies the \lext.
\end{thmalph}

\begin{thmalph}[Bogomolov--Sommese vanishing] \label{B}
Let $(X, D)$ be a complex-projective dlt \cC-pair and $\sA \subset \Sym_\cC^{[1]} \Omegal XrD = \Omegarl Xr{\rd D.}$ a rank one reflexive subsheaf.
Then the \cC-Kodaira dimension $\kappa_\cC(\sA) \le r$.
\end{thmalph}

The induction runs as follows, where the start of induction (dimension one) is trivial.
Here, of course, ``\cref A$_n$'' means ``\cref A for $X$ of dimension at most $n$'', and ditto for \cref B$_n$.
\begin{itemize}
\item \cref A$_n$ implies \cref B$_n$, and
\item \cref B$_n$ implies \cref A$_{n+1}$.
\end{itemize}
While the proofs of both directions do draw on some of the more elementary arguments in~\cite{GKKP11}, we stress that the technical core of that paper is not used.
Hence it still seems fair to say that our proof is ``new''.

\subsection{Background on \cC-pairs} \label{sec C-pairs}

For a more thorough treatment, see~\cite[Sec.~5]{Gra12} and the references therein.
A \emph{\cC-pair} is a pair $(X, D)$ in the usual sense, where $D = \sum_i (1 - \frac1{n_i}) D_i$ with $n_i \in \N \cup \set\infty$.
One then defines \emph{adapted morphisms} $\gamma \from Y \to X$, essentially by requiring that the ramification order over $D_i$ is equal to $n_i$.
The sheaves of \emph{\cC-differentials} are subsheaves
\[ \Sym_\cC^{[m]} \Omegal XrD \subset \big( \!\Sym^{[m]} \Omegap Xr \big) (* \rp D.) \]
defined by the condition that the pullback of a local section under an adapted morphism has at worst logarithmic poles along $\supp \gamma^* \rd D.$.
We have
\[ \Sym^{[m]} \Omegal Xr{\rd D.} \subset \Sym_\cC^{[m]} \Omegal XrD \]
and for $m = 1$, this is an equality.
Let $\sA \subset \Sym_\cC^{[1]} \Omegal XrD$ be a Weil divisorial subsheaf.
The \emph{\cC-Kodaira dimension} $\kappa_\cC(\sA)$ is defined to be the maximum of the dimensions of $\overline{\phi_m(X)}$, where $\phi_m$ is the rational map given by the global sections of $\Sym_\cC^{[m]} \sA$.
Here $\Sym_\cC^{[m]} \sA$ is the saturation of $\sA^{[m]}$ inside $\Sym_\cC^{[m]} \Omegal XrD$.

\subsection{Residue and restriction sequences}

We need to have the results of \cref{res dlt} at our disposal in this setting.
These are, to a large extent, already contained in~\cite[Sec.~11]{GKKP11} and~\cite[Sec.~6]{Gra12}.
Hence we only give a sketch of the proof.

\begin{thm}[Residue sequence] \label{res 0}
Let $(X, D)$ be a dlt \cC-pair and $P \subset \rd D.$ an irreducible component.
Setting $P^c \defn \Diff P{D-P}$, the pair $(P, P^c)$ is again a dlt \cC-pair, and the following holds:
For any integer $r \ge 1$, there is a sequence
\begin{equation} \label{res 0 seq}
0 \lto \Omegarl Xr{\rd D. - P} \lto \Omegarl Xr{\rd D.} \xrightarrow{\;\res_P\;} \Omegarl P{r-1}{\rd P^c.} \lto 0
\end{equation}
which is exact on $X$ off a codimension three subset and on $\snc{(X, D)}$ agrees with the usual residue sequence.
Its restriction to $P$ induces a sequence
\begin{equation} \label{res P 0 seq}
0 \lto \Omegarl Pr{\rd P^c.} \lto \Omegarl Xr{\rd D.} \big|_P \ddual \xrightarrow{\;\res_P^P\;} \Omegarl P{r-1}{\rd P^c.} \lto 0
\end{equation}
which is exact on $P$ off a codimension two subset.
More generally, for every $m \in \N$ there is a map
\[ \res_P^m \from \Sym_\cC^{[m]} \Omegal XrD \lto \Sym_\cC^{[m]} \Omegal P{r-1}{P^c}, \]
surjective off a codimension three subset of $X$, which generically coincides with the $m$-th symmetric power of the residue map. \qed
\end{thm}

\begin{thm}[Restriction sequence] \label{restr 0}
Notation as above.
Then there is a sequence
\begin{equation} \label{restr 0 seq}
0 \lto \Omegarl Xr{\rd D.}(-P) \ddual \lto \Omegarl Xr{\rd D. - P} \xrightarrow{\;\restr_P\;} \Omegarl Pr{\rd P^c.} \lto 0,
\end{equation}
exact off a codimension three subset, which on $\snc{(X, D)}$ agrees with the usual restriction sequence.
More generally, for every $m \in \N$ there is a map
\[ \restr_P^m \from \Sym_\cC^{[m]} \Omegal Xr{D-P} \lto \Sym_\cC^{[m]} \Omegal Pr{P^c} \]
which is surjective in codimension two and generically coincides with the $m$-th symmetric power of the restriction map. \qed
\end{thm}

\begin{proof}[Proof sketch of Theorems~\labelcref{res 0} and~\labelcref{restr 0}]
Sequences~\labelcref{res 0 seq} and~\labelcref{res P 0 seq} along with the respective properties are in~\cite[Thm.~11.7]{GKKP11}.
The existence of $\res_P^m$ is shown in~\cite[Thm.~6.9(i)]{Gra12}.
Likewise, the maps $\restr_P^m$ are constructed in~\cite[Thm.~6.9(ii)]{Gra12}.
What is missing is the following:
\begin{itemize}
\item the existence of sequence~\labelcref{restr 0 seq}, and
\item the surjectivity properties of $\res_P^m$ and $\restr_P^m$.
\end{itemize}
For the first item, the argument is similar to Step~3 in the proof of \cref{res}, i.e.~we reduce to the snc case by a reflexivity argument.
For the second item (which is in fact not used in this paper), we resort to Step~2 in the above proof, but instead of \cref{dlt quotient} we use~\cite[Prop.~6.12]{Gra12}.
\end{proof}

\subsection{Lifting along a non-positive map}

The analog of \cref{lift elem} is as follows.

\begin{thm}[Lifting forms] \label{lift elem 0}
Assume \cref B$_n$.
Let $g \from Y \to X$ be a proper birational map of normal varieties of dimension at most $n + 1$, with $E = \Exc(g)$ the reduced divisorial part of the exceptional locus.
Furthermore let $D$ be an effective divisor on $X$, and set $D_Y \defn g\inv_* D + E$.
Assume both of the following:
\begin{enumerate}
\item\label{lift elem 0.1} The pair $(Y, D_Y)$ is dlt and \Q-factorial.
\item\label{lift elem 0.2} For any irreducible component $P \subset E$, setting $P^c \defn \Diff P{D_Y - P}$, we have that $K_P + P^c$ is not $g\big|_P$-big.
\end{enumerate}
Then for any integer $r \ge 1$, the natural map
\[ g_* \Omegarl Yr{\rd D_Y.} \xhookrightarrow{\quad} \Omegarl Xr{\rd D.} \]
is an isomorphism.
\end{thm}

Recall that if $f$ is any map, a divisor on the source of $f$ is called $f$-big if its restriction to a \emph{general} fibre of $f$ is big.

The proof relies crucially on the following Negativity Lemma, which should be compared to the usual one~\cite[Lemma~3.6.2(1)]{BCHM10}.
Indeed our version is somewhat stronger, as it does not merely make a numerical statement, but actually produces sections of a suitable line bundle.

\begin{prp}[Big Negativity Lemma, cf.~\protect{\cite[Prop.~4.1]{Gra12}}] \label{neg lemma}
Let $\pi \from Y \to X$ be a proper birational map between normal quasi-projective varieties.
Then for any nonzero effective $\pi$-exceptional \Q-Cartier divisor $E$, there is an irreducible component $P \subset E$ such that $-E\big|_P$ is $\pi\big|_P$-big. \qed
\end{prp}

\begin{proof}[Proof of \cref{lift elem 0}]
We first contend that we may replace $D$ by $\rd D.$ and thus assume that $D$ is reduced.
To this end, note that $\rd D_Y. = g\inv_* \rd D. + E$, so the conclusion we are aiming at only depends on $\rd D.$.
Also, as $Y$ is \Q-factorial, the pair $(Y, \rd D_Y.)$ remains dlt.
Finally, for any component $P \subset E$ we have
\begin{align*}
\Diff P{\rd D_Y. - P} & = \Diff P0 + (\rd D_Y. - P)\big|_P && \text{\cite[(4.2.10)]{Kollar13}} \\
& = \Diff P0 + (D_Y - P)\big|_P - \{ D_Y \}\big|_P \\
& = P^c - \{ D_Y \}\big|_P && \text{\cite[(4.2.10)]{Kollar13}}
\end{align*}
where $\{ D_Y \} \defn D_Y - \rd D_Y. \ge 0$ is the fractional part of $D_Y$.
So Condition~\labelcref{lift elem 0.2} is likewise preserved.

Now let $\sigma \in \HH0.X.\Omegarl XrD. \setminus \set0$ be an arbitrary nonzero logarithmic $r$-form, and pick an effective $g$-exceptional divisor $G$ such that
\begin{equation} \label{rueckert 0}
g^* \sigma \in \HH0.Y.\Omegarl Yr{D_Y}(G) \ddual..
\end{equation}
Equivalently, there is an injective map $i \from \O Y(-G) \to \Omegarl Yr{D_Y}$.
We may assume that $G \ne 0$, in which case by \cref{neg lemma} there is a component $P \subset G$ such that $-G\big|_P$ is $g\big|_P$-big.
Set $P^c \defn \Diff P{D_Y - P}$.
Let $F \subset P$ be a general fibre of $g\big|_P$, and set $F^c \defn P^c\big|_F$.
Then $F$ is normal and $(F, F^c)$ is again a dlt \cC-pair.
Also, define $\sA \subset \Sym_\cC^{[1]} \Omegal Yr{D_Y}$ to be the image of $i$.
Now consider the residue sequence~\labelcref{res 0 seq} along $P$:
\[ \xymatrix{
& & \sA \ar@{ ir->}[d] \ar@/^.5pc/[dr] \ar@/_.5pc/@{-->}[dl] \\
0 \ar[r] & \Omegarl Yr{D_Y - P} \ar[r] & \Omegarl Yr{D_Y} \ar^-{\res_P}[r] & \Omegarl P{r-1}{\rd P^c.} \ar[r] & 0.
} \]

\begin{clm} \label{factor res 0}
We have $\res_P(\sA) = 0$, and hence $\sA$ is contained in $\Omegarl Yr{D_Y - P}$ as indicated by the dashed arrow in the above diagram.
\end{clm}

In the following, note that the restriction of a reflexive sheaf on $P$ to the general fibre $F$ remains reflexive and hence in this case the double dual may be omitted.

\begin{proof}[Proof of \cref{factor res 0}]
Arguing by contradiction, let us assume that $\res_P(\sA) \ne 0$ and denote its saturation by $\sB \subset \Sym_\cC^{[1]} \Omegal P{r-1}{P^c}$, a Weil divisorial sheaf.
By~\cite[Prop.~7.3]{Gra12}, there are a number $0 \le q \le r - 1$ and embeddings
\[ \alpha_k \from \sC_k \defn (\Sym_\cC^{[k]} \sB) \big|_F \xhookrightarrow{\quad} \Sym_\cC^{[k]} \Omegal Fq{F^c} \]
for all $k$, satisfying the compatibility conditions that $\sC_k$ and $\sC_1^{[k]}$ generically agree as subsheaves of $\Sym_\cC^{[k]} \Omegal Fq{F^c}$.
We will show that $\sC \defn \sC_1$ is ``\cC-big'' in the sense that $\kappa_\cC(\sC) = \dim F$.
If $q < \dim F$, this contradicts \cref B$_n$.
If $q = \dim F$, it contradicts Assumption~\labelcref{lift elem 0.2}, which says that $K_F + F^c = (K_P + P^c)\big|_F$ is not big.

To this end, we claim that for any natural number $m$ there is an inclusion
\begin{equation} \label{1771}
\big( \O Y(-mG)\big|_P \ddual \big) \big|_F \subset \big( (\Sym_\cC^{[m]} \sA)\big|_P \ddual \big) \big|_F \subset \Sym_\cC^{[m]} \sC.
\end{equation}
The first inclusion holds because $i$ does not vanish along $P$ nor $F$ (otherwise we would necessarily have $\res_P(\sA) = 0$).
For the second one, the map $\res_P^m$ from \cref{res 0} gives an inclusion
\[ (\Sym_\cC^{[m]} \sA) \big|_P \ddual \xhookrightarrow{\quad} \Sym_\cC^{[m]} \Omegal P{r-1}{P^c} \]
which by~\labelcref{gen equal} factors via the saturated subsheaf $\Sym_\cC^{[m]} \sB$.
Restricting to~$F$, we see that the middle term of~\labelcref{1771} is contained in $\sC_m$.
But $\sC_m$, $\sC^{[m]}$ and $\Sym_\cC^{[m]} \sC$ all generically agree.
Hence $\sC_m \subset \Sym_\cC^{[m]} \sC$ by another application of~\labelcref{gen equal} and we obtain~\labelcref{1771}.

Now let $m$ be sufficiently divisible so that $mG$ is Cartier.
In this case, on the left-hand side of~\labelcref{1771} the double dual may be dropped and we simply get the big line bundle $\O F \big( {-mG}\big|_F \big)$.
As a consequence, also $\Sym_\cC^{[m]} \sC$ is big, establishing our claim that $\kappa_\cC(\sC) = \dim F$.
\end{proof}

We next consider the restriction sequence~\labelcref{restr 0 seq}:
\[ \xymatrix{
& & \sA \ar@{ ir->}[d] \ar@/^.5pc/[dr] \ar@/_.5pc/@{-->}[dl] \\
0 \ar[r] & \Omegarl Yr{D_Y}(-P) \ddual \ar[r] & \Omegarl Yr{D_Y - P} \ar^-{\restr_P}[r] & \Omegarl Pr{\rd P^c.} \ar[r] & 0.
} \]
The same line of argument as in the proof of \cref{factor res 0} then shows that we have $\restr_P(\sA) = 0$ and so $\sA$ is contained in $\Omegarl Yr{D_Y}(-P) \ddual$.
The details are omitted for their similarity.
The proof of \cref{lift elem 0} is now finished in exactly the same fashion as \cref{lift elem}: we have shown that $G$ can be replaced by $G - P$ in~\labelcref{rueckert 0}.
Hence after finitely many repetitions we arrive at $G = 0$.
\end{proof}

\subsection{Proof of ``\labelcref A$_n$ \imp\ \labelcref B$_n$''}

This implication is by far the easier of the two.
By~\cite{BCHM10}, the dlt pair $(X, D)$ admits a \Q-factorialization.
This is essentially a dlt modification in the special case of dlt pairs, cf.~\cref{sec fact res}.
For the proof, see~\cite[Thm.~3.1]{KK10} and~\cite[Thm.~9.2]{Gra12}.

As the Kodaira dimension is invariant under small morphisms, we may assume that~$\sA$ is \Q-Cartier.
Under this assumption, the proof of~\cite[Thm.~7.2]{GKKP11} shows how to deduce that $\kappa(\sA) \le r$ from \cref A$_n$ and the standard Bogomolov--Sommese vanishing theorem for snc pairs~\cite[Cor.~6.9]{EV92}.
The stronger statement that $\kappa_\cC(\sA) \le r$ can then be obtained from this by a branched covering trick as explained in~\cite[Sec.~7]{JK11}.
Compare also~\cite[Thm.~7.3]{GKKP11} (which erroneously contains the extra assumption that $\dim X \le 3$). \qed

\subsection{Proof of ``\labelcref B$_n$ \imp\ \labelcref A$_{n+1}$''}

Let $\pi \from Y \to X$ be a log resolution of $(X, D)$.
Then $\pi$ can be factored as $f \circ \phi_{r-1} \circ \cdots \circ \phi_0$ just as in \cref{fact res}, whose notation we adopt here.
The only difference is that some of the $\phi_i$ might be flips.
Also the proof is essentially the same, except that instead of~\cite{Tanaka18} we need to appeal to~\cite{BCHM10}.
Also, because we cannot in general run the MMP on a dlt pair, we have to use the perturbation trick from the proof of~\cite[Thm.~3.1]{KK10} to reduce the situation to the case of klt pairs.

We will lift forms along each step separately, just as in \cref{tame lext} but using \cref{lift elem 0} instead of \cref{lift elem}.
For this, we need to make sure Condition~\labelcref{lift elem 0.2} is satisfied.
As far as the map $f$ is concerned, this is quite clear: since $K_Z + \wt D_r = f^* (K_X + D)$, the restriction of $K_Z + \wt D_r$ to any fibre of $f$ is trivial and in particular not big.
But for any component $P \subset E$, we have $(K_Z + \wt D_r)\big|_P = K_P + P^c$ by adjunction and so~\labelcref{lift elem 0.2} is satisfied.

We now fix $0 \le i \le r - 1$ and turn to the map $\phi_i \from Y_i \to Y_{i+1}$.
If $\phi_i$ is a flip, then it is an isomorphism in codimension one and hence extension of forms from $Y_{i+1}$ to $Y_i$ is automatic by reflexivity.
We may therefore assume that $\phi_i$ is a divisorial contraction, with irreducible exceptional divisor $P$.
By construction, $K_{Y_i} + \wt D_i$ is $\phi_i$-anti-ample.
Also, $(K_{Y_i} + \wt D_i)\big|_P = K_P + P^c$ by adjunction.
Thus $K_P + P^c$ is $\phi_i\big|_P$-anti-ample and in particular~\labelcref{lift elem 0.2} is satisfied. \qed

\section{Sharpness of results} \label{sharp}

In this section we discuss some examples that show to what extent our main results are optimal.
First, we show that \cref{main lc} fails for the rational double point $E_8^0$ in characteristic $p \le 5$, using notation from~\cite{Artin77}.

\begin{exm}[No \lext in low characteristics] \label{E8 sharp}
Fix any field $k$ of characteristic $p = 2, 3 \text{ or } 5$.
Then for the (non-$F$-pure) $E_8^0$ singularity
\[ X = \set{f = z^2 + x^3 + y^5 = 0} \subset \A3_k, \]
the \lext does not hold.
More precisely, note that \kahler differentials on $X$ satisfy the relation
\[ \d f = 3x^2 \d x + 5y^4 \d y + 2z \d z = 0 \]
and hence we may consider the $1$-form
\[ \sigma \defn \begin{cases}
y^{-4} \d x = \phantom- x^{-2} \d y,   & p = 2, \\
z^{-1} \d y = - y^{-4} \d z, & p = 3, \\
z^{-1} \d x = - x^{-2} \d z, & p = 5.
\end{cases}
\]
As any two coordinate functions on $X$ vanish simultaneously only at the origin, $\sigma \in \HH0.\Reg X.\Omegap X1. = \HH0.X.\Omegar X1.$ is a reflexive differential form on $X$.
We blow up the origin of $\A3_k$ (and points lying over it) four times in a row, yielding a map $\phi \from \wt{\A3_k} \to \A3_k$.
In suitable coordinates on $\wt{\A3_k}$, this map is given by
\[ \phi(u, v, w) = (u^2 v^5, uv^3, u^2 v^7 w). \]
We compute
\[ \phi^*(f) = u^4 v^{14} w^2 + u^6 v^{15} + u^5 v^{15} = u^4 v^{14} \cdot \underbrace{\big( w^2 + uv(u + 1) \big)}_{\substack{\text{equation of strict} \\[.3ex] \text{transform $\wt X$ of $X$}}}. \]
We see that $\wt X$ can be parametrized rationally by the $(u, w)$-plane, namely by setting $v = - \frac{w^2}{u(u + 1)}$.
In this parametrization, for $p = 2$ the pullback of $\sigma$ is given by
\begin{align*}
\phi^*(\sigma) & = (uv^3)^{-4} \, \d (u^2 v^5) \\
& = u^{-2} v^{-8} \, \d v \\
& = \frac{u^6 (u + 1)^8}{w^{16}} \, \d \left( \frac{w^2}{u(u + 1)} \right) \\
& = \cdots \\
& = \frac{u^4 (u + 1)^6}{w^{14}} \, \d u.
\end{align*}
A similar calculation for the other characteristics gives
\[ \phi^*(\sigma) = \begin{cases}
\displaystyle \frac{u^2 (u + 1)^4}{w^9} \, \d u, & p = 3, \\[2.5ex]
\displaystyle \frac{u (u + 1)^2}{w^5} \, \d u, & p = 5.
\end{cases}
\]
This shows that the extension of $\sigma$ to $\wt X$ has worse than logarithmic poles along the exceptional divisor $\set{w = 0}$.
\end{exm}

Next, we show that in \cref{main reg}, the assumption on $\det(E_i \cdot E_j)$ not being divisible by $p$ really is necessary.

\begin{exm}[No \et in spite of \lext] \label{lext no et}
Let $k$ be a field of characteristic $p > 0$, and consider the $p$-th Veronese subring $R$ of $k[x, y]$, that is, $R = k[x^p, x^{p-1}y, \dots, y^p]$.
Then $X = \Spec R$ is a strongly $F$-regular surface, since $R$ is a direct summand of the regular ring $k[x, y]$.
In particular, $X$ is klt.
If $\pi \from Y \to X$ is the minimal resolution, then $E = \Exc(\pi)$ consists of a single smooth rational curve of self-intersection $-p$.
In particular, the assumptions of \cref{main reg} are not satisfied.
For later use, let us record the discrepancy $a = a(E, X)$ along $E$: by adjunction,
\[ -2 = (K_Y + E) \cdot E = \big( \pi^* K_X + (a + 1) E \big) \cdot E = - (a + 1) p \]
and hence $a = -1 + \frac2p$.

One can see by direct calculation that $X$ satisfies the \lext, but not the \et.
More precisely, $Y$ is covered by two open sets $U_0, U_1 \isom \A2_k$, where $U_i$ has coordinates $x_i, y_i$ and the coordinate change is given by $(x_1, y_1) = (x_0\inv, x_0^p y_0)$.
The exceptional curve $E$ is given by the equation $y_i = 0$ in the chart $U_i$.
Consider the form $\sigma \in \HH0.Y \setminus E.\Omegap Y1.$ given by $y_0\inv \d y_0$ on $U_0$ and by $y_1\inv \d y_1$ on $U_1$.
It obviously does not extend regularly over $E$, showing that the \et fails for $X$.
But $\sigma$ has only a logarithmic pole along $E$, and it generates the quotient $\factor{\Omegar X1}{\pi_* \Omegap Y1}$.
Thus the \lext does hold for $X$.
(Of course, this last fact also follows from \cref{main lc}, at least if $p \ge 7$.)
\end{exm}

An elaboration of the previous example shows that \cref{restr} fails for dlt pairs whose canonical divisor is not \zp-Cartier.

\begin{exm}[No restriction sequence for fiercely dlt pairs] \label{no zp no restr}
Using notation from \cref{lext no et}, let $D \subset X$ be the $\pi$-image of the curve $D_Y \subset U_0 \subset Y$ given by the equation $\set{x_0 = \text{const.}}$ (where the constant is arbitrary).
Then $D$ is a smooth curve passing through the singular point $x \in X$, and it is isomorphic to its strict transform $D_Y$.
\[ \xymatrix{
D_Y \ar@{ ir->}[rr] \ar_-\wr[d] & & Y \ar^-\pi[d] \\
D \ar@{ ir->}[rr] & & X
}
\]
We have $\pi^* D = D_Y + rE$, where $0 = \pi^* D \cdot E = 1 - rp$ and hence $r = 1/p$.
It follows that the discrepancy of $(X, D)$ along $E$ is
\[ a(E, X, D) = a(E, X) - r = - \frac{p-1}p. \]
This shows that the pair $(X, D)$ is plt.
On the other hand, $K_X + D$ cannot be \zp-Cartier, since otherwise $p$ would not appear in the denominator of $a(E, X, D)$ (written in lowest terms).
Hence $(X, D)$ is fiercely dlt.
In particular, we cannot apply \cref{compute diff} to compute $\Diff D0$.
But~\cite[Thm.~3.36]{Kollar13} tells us that we still have $D^c \defn \Diff D0 = \big( 1 - \frac1p \big) \cdot [x]$.
So, if \cref{restr} held, we would have a restriction map as in~\labelcref{restr seq}
\[ \restr_D \from \Omegar X1 \to \Omegal D1{\rd D^c.} = \Omegap D1. \]
Consider however the form $\sigma$ from the previous example, viewed as a section of $\Omegar X1$.
As $D_Y \isom D$, we can compute $\restr_D(\sigma)$ on $Y$.
We have already seen that $\sigma$ acquires a logarithmic pole along $E$.
So $\sigma\big|_{D_Y}$ has a logarithmic (i.e.~simple) pole at the unique point in the intersection $D_Y \cap E$, which under $\pi$ maps to $x$.
Summing up, this means that we do have a restriction map
\[ \restr_D \from \Omegar X1 \to \Omegal D1x, \]
but it does not factor via $\Omegap D1$.
Looking at higher powers $\sigma^{[m]}$, we see that also the other maps $\restr_D^m$ from \cref{restr} do not exist in this example.
\end{exm}

\begin{exm}[No \lext for singularities reduced from characteristic zero]
Finally, we would like to remark that if we start with a log canonical singularity in characteristic zero and then reduce it modulo some small prime $p$, the resulting singularity may not satisfy the \lext even if it remains log canonical.
Indeed, \cref{E8 sharp} furnishes a counterexample since $z^2 + x^3 + y^5 = 0$ defines an $E_8$ rational double point also in characteristic zero.
\end{exm}

\section{Counterexamples in higher dimensions} \label{ext higher-dim p}

In this final section, we will prove \cref{lext p fail}.
As a starting point, in~\cite{Kol95nh} Koll\'ar has given a fairly explicit method for constructing counterexamples to Bogo\-molov--Sommese vanishing over fields of positive characteristic.
We will recall Koll\'ar's construction in \cref{kollar construction} below, both for the benefit of the reader and in order to bring the result in the precise form we need.
It turns out that in the examples, the line bundle in question is not just big, but even ample.
Thus also Nakano vanishing is violated:

\begin{prp}[Failure of Nakano vanishing] \label{nakano fail}
Fix an algebraically closed field $k$ of characteristic~$p > 0$, and an integer $n \ge 2$.
\begin{enumerate}
\item\label{Nakano Fano} If $n \ge 2p - 2$, then there exists an $n$-dimensional Fano variety $Y/k$ with only isolated canonical hypersurface singularities such that
\[ \HH0.Y.\Omegar Y{n - 1} \tensor \can Y. \ne 0. \]
\item\label{Nakano Fano II} If $n \ge 3p - 2$, then there exists $Y$ as above, but such that $\can Y\inv$ admits a square root $L$ (i.e.~an ample line bundle with $L^2 \isom \can Y\inv$) and
\[ \HH0.Y.\Omegar Y{n - 1} \tensor L\inv. \ne 0. \]
\item\label{Nakano CY} If $n \ge p - 2$, then there exists an $n$-dimensional variety $Y/k$ with only isolated canonical hypersurface singularities, satisfying $\can Y \isom \O Y$ and
\[ \HH0.Y.\Omegar Y{n - 1} \tensor L\inv. \ne 0 \]
for some ample line bundle $L$ on $Y$.
\end{enumerate}
In all cases, $Y$ actually has $F$-pure singularities.
If $n \ge 3$, then $Y$ is even terminal and strongly $F$-regular.
\end{prp}

In \cref{ext cones}, we will turn our attention to cones over projective varieties and study when the \lext holds for such spaces.
The conclusion is that cones over the examples from \cref{nakano fail} are sufficient to prove \cref{lext p fail}, which is accomplished in \cref{proof lext p fail}.

\subsection{Koll\'ar's construction} \label{kollar construction}

Koll\'ar's method is quite flexible in the sense that it does not rely on resolution of singularities and gives very good control on the canonical divisor of the resulting example.
On the other hand, it only works in dimensions that satisfy a certain lower bound depending on the characteristic, and the spaces obtained are virtually never smooth.
Also, the violation of Nakano vanishing is only guaranteed in degree $n - 1$, where $n$ is the dimension.

Let $X$ be an $n$-dimensional smooth projective variety over an algebraically closed field of characteristic $p > 0$ and $L$ a line bundle on $X$.
Assume that $L^p$ is ``globally generated to second order'' in the sense that the restriction map
\[ \HH0.X.L^p. \lto L^p \tensor_{\O X} \!\left( \!\factor{\O X}{\frm_x^3} \right) \]
is surjective for every (closed) point $x \in X$ with ideal sheaf $\frm_x \subset \O X$.
Choose a general section $s \in \HH0.X.L^p.$ and consider the cover
\[ Y \defn X[{\sqrt[p]s}] \xrightarrow{\hspace{.75em}\pi\hspace{.75em}} X \]
as before.
By~\cite[(14.2)]{Kol95nh} there is a short exact sequence
\begin{equation} \label{1901}
0 \lto \pi^* \coker \underbrace{\left[ L^{-p} \xrightarrow{\d s} \Omegap X1 \right]}_{\substack{\text{const.~rank one} \\[.3ex] \text{off small subset}}} \lto \Omegap Y1 \lto \pi^* L\inv \lto 0.
\end{equation}
Taking determinants, we see that $K_Y = \pi^* \big( K_X + (p - 1) L \big)$.
On the other hand, the $(n - 1)$-th reflexive wedge power of the first map in~\labelcref{1901} shows that
\begin{equation} \label{1921}
\HH0.Y.\Omegar Y{n - 1} \tensor \pi^* (K_X + pL)\inv. \ne 0.
\end{equation}
Thus we obtain interesting examples if $K_X + pL$ is ample, but $K_X + (p - 1)L$ is not.

\begin{proof}[Proof of \cref{nakano fail}]
Let $X \subset \PP{n + 1}$ be a smooth hypersurface of degree $d = n - 2p + 3 \ge 1$, and take $L = \O X(2)$.
The global generation hypothesis on $L^p$ is automatically satisfied, hence we may construct $\pi \from Y \to X$ as above.
Then $Y$ is Fano since
\[ K_X + (p - 1) L = \O X \big( \! - \! (n + 2) + d + 2(p - 1) \big) = \O X( - 1) \]
is anti-ample.
On the other hand, $K_X + pL = \O X(1) = \big( K_X + (p - 1) L \big)\inv$ is ample and so by~\labelcref{1921}, the variety $Y$ violates Nakano vanishing in the required form.
By~\cite[(20.3), (22.1)]{Kol95nh}, the singularities of $Y$ are locally of the form
\begin{equation} \label{1880}
y^p = \underbrace{x_{n - 1} x_n + f_2(x_1, \dots, x_{n - 2})}_{\text{non-degenerate quadric}} \; + \; \text{(higher-order terms w.r.t.~$x$)}.
\end{equation}
Using this description, it can be checked that $Y$ has only isolated canonical hypersurface singularities, which are terminal for $n \ge 3$.
This proves~\labelcref{Nakano Fano}.

The argument for~\labelcref{Nakano Fano II} is very similar.
We start with $X$ a smooth hypersurface of degree $d = n - 3p + 3 \ge 1$ and $L = \O X(3)$.
Then
\[ K_X + (p - 1) L = \O X \big( \! - \! (n + 2) + d + 3(p - 1) \big) = \O X( - 2) \]
and $K_X + pL = \O X(1)$.
Again we conclude by~\labelcref{1921}.

For~\labelcref{Nakano CY}, we tweak the numbers once more.
Let $X$ be of degree $d = n - p + 3 \ge 1$ and $L = \O X(1)$.
Then
\[ K_X + (p - 1) L = \O X \big( \! - \! (n + 2) + d + (p - 1) \big) = \O X, \]
so $\can Y \isom \O Y$, and $K_X + pL = \O X(1)$.

The claim about $F$-purity can likewise be checked using~\labelcref{1880} and Fedder's criterion~\cite{Fedder83}.
If $n \ge 3$, then even~\labelcref{1880} multiplied by the non-unit $x_1$ is $F$-pure and so $Y$ is strongly $F$-regular.
Note also that a strongly $F$-regular Gorenstein singularity is automatically canonical, therefore this provides an alternative proof of $Y$ being canonical.
\end{proof}

\begin{rem}
One might be tempted to try and construct lower-dimensional examples by starting with a more interesting $X$ than just a hypersurface in $\PP{n + 1}$.
This, however, is not possible because the Fano index of $X$ is always $\le \dim X + 1$ by~\cite[Ch.~V, Thm.~1.6]{Kollar96}.
\end{rem}

\subsection{Extension properties on cones} \label{ext cones}

Fix an integer $n \ge 2$, a smooth projective variety $Y$ with $\dim Y = n - 1$, and an ample line bundle $L$ on $Y$.
Following~\cite[Ch.~3.1]{Kollar13}, let
\[ X \defn \Spec \bigoplus_{m \ge 0} \HH0.Y.L^m. \]
be the affine cone over $(Y, L)$.
Blowing up the vertex gives a log resolution $\pi \from \wt X \to X$, where $\wt X$ is the total space of the line bundle $L\inv$ and the exceptional locus $E$ is the zero section of $L\inv$.
In particular, there is an affine map $r \from \wt X \to Y$, which maps $E$ isomorphically onto $Y$.

For any integer $q \ge 0$, we will say that Condition~\labelcref{kan}$_q$ holds if
\begin{equation} \label{kan}
\tag{$*$}
\HH0.Y.\Omegap Yq \tensor L^{-m}. = 0 \qquad \text{for all $m \ge 1$.}
\end{equation}
Note that~\labelcref{kan}$_q$ always holds in any of the following cases: $q = 0$, $q \ge n$, or if $L$ is sufficiently ample.
In characteristic zero, \labelcref{kan}$_q$ holds for any $q \ne n - 1$ by Nakano vanishing.

With this notation in place, we have the following result.
It should be compared to the non-logarithmic, characteristic zero version in~\cite[Prop.~B.2]{KebekusSchnell18}.

\begin{prp}[\lext on cones] \label{lext cones}
Notation as above. Then the following equivalences hold:
\begin{enumerate}
\item $X$ satisfies the \lext for $1$-forms \gdw~\labelcref{kan}$_1$ holds.
\item $X$ satisfies the \lext for $n$-forms \gdw~\labelcref{kan}$_{n-1}$ holds.
\end{enumerate}
More generally, for arbitrary values of $q$ we have the following:
\begin{enumerate}
\item If~\labelcref{kan}$_q$ and~\labelcref{kan}$_{q-1}$ hold, then $X$ satisfies the \lext for $q$-forms.
\item\label{1968} Conversely, if $X$ satisfies the \lext for $q$-forms, then~\labelcref{kan}$_q$ holds.
If in addition $n \ge 3$ and $L$ is sufficiently ample, then also~\labelcref{kan}$_{q-1}$ holds.
\end{enumerate}
\end{prp}

\begin{proof}
The sequence of relative differentials for $r$ reads
\begin{equation} \label{2001}
0 \lto r^* \Omegap Y1 \lto \Omegap{\wt X}1 \lto r^* L \lto 0
\end{equation}
and its logarithmic version is
\begin{equation} \label{2003}
0 \lto r^* \Omegap Y1 \lto \Omegal{\wt X}1E \lto r^* L (E) \lto 0.
\end{equation}
For~\labelcref{2003}, choose a system of local parameters $y_1, \dots, y_{n-1}$ of $Y$ and let $t$~be a nowhere vanishing local section of $L$, considered as a fibrewise linear coordinate on~$\wt X$.
Then the middle term of~\labelcref{2003} is locally freely generated by
\[ r^* \d y_1, \, \dots, \, r^* \d y_{n-1}, \, \d t/t, \]
and the first map is the inclusion of the subsheaf generated by the $r^* \d y_i$.
Consequently, the quotient sheaf is invertible, locally generated by $\d t /t$.
Since $E = \set{ t = 0 }$, this shows that the quotient is isomorphic to $r^* L (E)$.

Now~\labelcref{2001} and~\labelcref{2003} sit inside the diagram shown in Figure~\labelcref{1926} \vpageref{1926}.
\begin{figure}[!t]
\centerline{
\xymatrix@R=4ex{
& & 0 \ar[d] & 0 \ar[d] \\
0 \ar[r] & r^* \Omegap Y1 \ar[r] \ar@{=}[d] & \Omegap{\wt X}1 \ar[r] \ar[d] & r^* L \ar[d] \ar[r] & 0 \\
0 \ar[r] & r^* \Omegap Y1 \ar[r] & \Omegal{\wt X}1E \ar[r] \ar^-{\res_E}[d] & r^*L (E) \ar[r] \ar[d] & 0 \\
& & \O E \ar@{=}[r] \ar[d] & \O E \ar[d] \\
& & 0 & 0
} }
\caption{Relative (log) differential sequences for the map $r$.}\label{1926}
\end{figure}
Also, for forms of higher degree, from~\labelcref{2003} we get~\cite[Ch.~II, Ex.~5.16]{Har77}
\begin{equation} \label{1920}
0 \lto r^* \Omegap Yq \lto \Omegal{\wt X}qE \lto r^* \big( \Omegap Y{q-1} \tensor L \big) (E) \lto 0.
\end{equation}
Recalling that both $r$ and its restriction $r' \defn r\big|_{\wt X \setminus E}$ are affine, with
\begin{align*}
r_* \O{\wt X} & = \ \bigoplus_{m \ge 0} L^m, \\
r_* \O{\wt X}(E) & = \bigoplus_{m \ge -1} L^m, \quad \text{and} \\
r'_* \O{\wt X \setminus E} & = \ \bigoplus_{m \in \Z} L^m,
\end{align*}
from~\labelcref{1920} we obtain the following diagram with exact rows and injective vertical arrows:
\[ \begin{tikzpicture}[baseline= (a).base]
\node[scale=0.845] (a) at (0,0){
\begin{tikzcd}[column sep=small,row sep=small]
	0 \rar & \displaystyle \bigoplus_{m \ge 0} \mathrm H^0(\Omegap Yq \tensor L^m)
		\rar \dar[hookrightarrow,start anchor={[yshift=2.2ex]},end anchor={[yshift=-0.3ex]}]{\alpha}
			& \HH0.\wt X.\Omegal{\wt X}qE.
		\rar \dar[hookrightarrow,end anchor={[yshift=-.55ex]}]{\beta} & \displaystyle \bigoplus_{m \ge 0} \mathrm H^0(\Omegap Y{q-1} \tensor L^m)
		\rar \dar[hookrightarrow,start anchor={[yshift=2.2ex]},end anchor={[yshift=-0.3ex]}]{\gamma}
			& \displaystyle \bigoplus_{m \ge 0} \mathrm H^1(\Omegap Yq \tensor L^m)
		\dar[hookrightarrow,start anchor={[yshift=2.2ex]},end anchor={[yshift=-0.3ex]}]{\delta} \\
	0 \rar & \displaystyle \bigoplus_{m \in \Z} \mathrm H^0(\Omegap Yq \tensor L^m)
		\rar & \HH0.\wt X \setminus E.\Omegap{\wt X}q.
		\rar & \displaystyle \bigoplus_{m \in \Z} \mathrm H^0(\Omegap Y{q-1} \tensor L^m)
		\rar & \displaystyle \bigoplus_{m \in \Z} \mathrm H^1(\Omegap Yq \tensor L^m)
\end{tikzcd}
};
\end{tikzpicture}
\]
It is clear that $\alpha$ is an isomorphism \gdw~\labelcref{kan}$_q$ holds, $\beta$ is an isomorphism \gdw\ the \lext for $q$-forms holds on $X$, and $\gamma$ is an isomorphism \gdw~\labelcref{kan}$_{q-1}$ holds.
Furthermore, if $n \ge 3$ and $L$ is sufficiently ample then $\delta$ is an isomorphism by Serre vanishing and Serre duality.
All claims thus follow from straightforward diagram chases (cf.~\cite[Lemma~B.2]{GKKP11}).
\end{proof}

\subsection{Proof of \cref{lext p fail}} \label{proof lext p fail}

With all preliminaries in place, the construction of counterexamples to the \lext becomes very easy.
Take $(Y, L)$ as in~\labelcref{Nakano CY}, and let $X$ be the affine cone over $(Y, L)$.
Blowing up the vertex gives an exceptional divisor of discrepancy $-1$ because $\can Y \isom \O Y$.
The result is the total space of $L\inv$, which has canonical singularities just as $Y$.
We conclude that $X$ is log canonical.
By~\labelcref{1968}, the \lext for $(n - 2)$-forms does not hold on $X$.
This proves~\labelcref{lext p fail lc}.

For~\labelcref{lext p fail can}, we use the Fano variety $Y$ from~\labelcref{Nakano Fano} instead.
In this case, $X$ is the cone over $(Y, \can Y\inv)$.
A calculation shows that the first discrepancy is zero.
Hence, since $Y$ has canonical singularities, so does $X$.
The \lext fails for the same reason as above.

For~\labelcref{lext p fail term}, we appeal to~\labelcref{Nakano Fano II}, i.e.~the cone $X$ is taken with respect to a square root of $\can Y\inv$.
In this case the first discrepancy is equal to one.
Since $\dim Y \ge 3p - 2 \ge 4$, we know that $Y$ has only terminal singularities and then the same is true of $X$.

In each case, a log resolution of $X$ can be obtained by first blowing up the vertex of the cone and then pulling back everything along a resolution of $Y$, which exists by~\cite[\S 21]{Kol95nh}. \qed

\providecommand{\bysame}{\leavevmode\hbox to3em{\hrulefill}\thinspace}
\providecommand{\MR}{\relax\ifhmode\unskip\space\fi MR}
\providecommand{\MRhref}[2]{%
  \href{http://www.ams.org/mathscinet-getitem?mr=#1}{#2}
}
\providecommand{\href}[2]{#2}

\end{document}